\documentclass{article}

\usepackage[T2C]{fontenc}
\usepackage[cp866]{inputenc}
\usepackage[english]{babel}
\usepackage[tbtags]{amsmath}
\usepackage{amsfonts,amssymb,mathrsfs,amscd,amsmath}
\usepackage{tabularx, longtable}
\usepackage{rotating}
\usepackage[all]{xy}

\advance\hoffset by -1in
\advance\voffset by -1in

\newtheorem{theorem}{Theorem}
\newtheorem{lemma}{Lemma}
\newtheorem{proposition}{Proposition}
\newtheorem{corollary}{Corollary}
\newtheorem{definition}{Definition}
\newtheorem{proof}{Proof}

\newcommand*{\cd}{(\cdot)}

\newcommand*{\wx}{\widehat x}

\newcommand*{\la}{\langle}
\newcommand*{\ra}{\rangle}
\newcommand*{\wu}{\widehat u}

\newcommand*{\wou}{\widehat{\overline u}}
\newcommand*{\woa}{\widehat{\overline \alpha}}
\newcommand*{\ov}{\overline}
\newcommand*{\ws}{\widehat \sigma}

\newcommand*{\wy}{\widehat y}

\newcommand*{\wo}{\mathcal O}
\newcommand*{\ww}{\widehat w}
\newcommand*{\wF}{\widehat F}

\let\le=\leqslant
\let\ge=\geqslant
\def\bad{\spaceskip=0.33emplus0.6emminus0.15em\immediate\write5{\string\bad}}

\title{Local infimum in optimal control}
\author{E.~R.~Avakov, G.~G.~Magaril-Il'yaev}
\date{}

\begin{document}

\maketitle

\begin{abstract}
The concept of a local infimum for an optimal control problem is introduced. This definition
extends that of an optimal process. For a~local infimum we prove an existence theorem
and derive necessary conditions that resemble some family of ``maximum principles''.
Examples are given to demostrate the meaningfulness of the
necessary conditions obtained in the present paper, which extend and strengthen the classical results in this field.
\end{abstract}

\section*{Introduction}

By the Pontryagin maximum principle for an optimal control problem, one means, as is well known
(see \cite{P}), necessary conditions for optimality of a~process\,---\,this being a~pair
consisting of an optimal control and the corresponding optimal (phase)
trajectory. In the present paper, we introduce the concept of a~local infimum, which extends
that of an optimal trajectory. For a~local infimum, necessary conditions
which resemble some family of ``maximum principles'' are proved. If a~local infimum is, in particular,
an optimal trajectory, then this family contains the classical Pontryagin maximum principle,
as well as some other relations, which in general are capable of providing additional information
and thereby, as is shown by examples, strengthen the Pontryagin maximum principle.

If a local infimum is not an optimal trajectory, then these necessary conditions provide a~tool for
finding trajectories  ``suspicious'' for a~local
infimum. The use of this machinery is to a large extent the same as that of the Pontryagin maximum principle for
finding processes which are suspicious for optimality.

In the present paper, we employ the idea of ``convexification'' of the original control problem. This
idea can be found in the book of Gamkrelidze~\cite{Gam}. We also use this idea, but in
a~less general setting, which, however, is quite sufficient for our purposes.

We also put forward an existence theorem for a local infimum in an
optimal control problem. This result does not involve any convexity-type assumptions.
So, the class of problems in which there exists a~local infimum turns  out to be considerably wider than the
class of problems in which one can guarantee the existence of an optimal process, because the existence
of the latter depends on fairly inconvenient assumptions on the convexity of some family of
sets associated with the problem.

The paper consists of three sections. In the first section we formulate and prove the main
results. In the second section we give examples illustrating the results obtained.
In the third section (the Appendix) we prove a~generalized
implicit function theorem and derive four lemmas: the inverse function lemma, the lemma on equation in variations,
and two approximation lemmas.
All these results, which in our opinion are of independent interest, are chief ingredients in the
proofs of the main results of the paper.

\smallskip

The authors are deeply grateful to Revaz Valer'yanovich Gamkrelidze for
his kind attention and useful discussions.

\section*{Main results and proofs}

Let $U$ be a nonempty subset of $\mathbb R^r$. Assume that we are given
a~mapping  $\varphi\colon
\mathbb R\times\mathbb R^n\times\mathbb R^r\to\mathbb R^n$ of variables
$t$, $x$ and~$u$, a~function $f_0\colon\mathbb R^n\times\mathbb R^n\to \mathbb R$, and mappings
$f\colon\mathbb R^n\times\mathbb R^n\to \mathbb R^{m_1}$ and $g\colon\mathbb
R^n\times\mathbb R^n\to \mathbb R^{m_2}$ of variables $\zeta_i\in\mathbb R^n$, $i=1,2$.

Consider the following optimal control problem
\begin{align}
& f_0(x(t_0),x(t_1))\to\inf, \label{1}\\[3pt]
&\dot x=\varphi(t,x,u(t)),\quad u(t)\in U\,\,\text{for almost all}\,\,\, t\in[t_0,t_1],\label{2}\\[3pt]
& f(x(t_0),x(t_1))\le0,\quad g(x(t_0),x(t_1))=0\label{3}.
\end{align}

It what follows it will be assumed that the {\it mapping $\varphi$ is continuous together with its
derivative with respect to~$x$ on $\mathbb R\times\mathbb R^n\times \mathbb R^r$ and the mappings
$f_0$, $f$ and $g$ are continuously differentiable on $\mathbb R^n\times\mathbb R^n$}.

In connection with the definition of a local infimum,
it will be convenient to rephrase the standard definitions of an admissible process and
an optimal processes with emphasis on the concept of a~trajectory
and when the control plays a~secondary role.

\begin{definition}\label{d1}\rm
A trajectory $x\cd\in AC([t_0,t_1],\mathbb R^n)$ is called \textit{admissible} for the control system
\eqref{2},\,\eqref{3} if there exists a~control $u\cd\in
L_\infty([t_0,t_1],\mathbb R^r)$ satisfying conditions \eqref{2},\,\eqref{3}.
\end{definition}

The set of admissible trajectories for the control system  \eqref{2},\,\eqref{3} (in what follows,
for brevity, we shall drop the word ``control'') will be always considered as
a~subset of $C([t_0,t_1],\mathbb R^n)$.

\begin{definition}\label{d2}\rm
An admissible trajectory $\wx\cd$ is called an optimal trajectory for problem
\eqref{1}--\eqref{3} if it delivers a~local minimum to the functional $f_0$ on
the set of admissible trajectories.
\end{definition}

The following definition is the main definition  in the present paper.

\begin{definition}\label{d3}\rm
A function $\wx\cd\in C([t_0,t_1],\mathbb R^n)$ is called a~\textit{local infimum} for problem
\eqref{1}--\eqref{3} if it delivers a~local minimum for the functional $f_0$  on the closure of the set
of admissible trajectories.

If a minimum is global, then we speak about the global infimum.
\end{definition}

It is clear that the value of $f_0$ on a global infimum coincides with the infimum of $f_0$ over
all admissible trajectories.

It is easily seen that if $\wx\cd$ is an optimal trajectory for problem \eqref{1}--\eqref{3},
then $\wx\cd$ is a~local infimum for this problem. On the other hand, if a~function $\wx\cd$
is a~local infimum for problem \eqref{1}--\eqref{3} and is admissible, then $\wx\cd$ is an
optimal trajectory for this problem.

The above definition of local infimum is not only formally more general than that
of an optimal trajectory, but it also has an advantage over the definition of an optimal trajectory
in that the class of optimal control problems in which the existence of a~local infimum can be guaranteed
is considerably larger than the class of problems in which one secures the existence of an optimal
trajectory, because in the former case it is not required to satisfy the cumbersome condition that the
set
\begin{equation}\label{4}
\varphi(t,x,U)=\{\,\varphi(t,x,u)\in\mathbb R^n : u\in U\,\}
\end{equation}
be convex for all $t\in[t_0,t_1]$ and $x\in\mathbb R^n$.

The existence theorem will be proved for the following particular case of problem
\eqref{1}--\eqref{3}:
\begin{align}
& f_0(x(t_1))\to\inf,\label{5} \\[3pt]
&\dot x=\varphi(t,x,u(t)),\quad u(t)\in U\,\,\text{for almost all}\,\,\, t\in[t_0,t_1],\label{6}\\[3pt]
& x(t_0)=x_0,\quad g(x(t_1))=0;\label{7}
\end{align}
here $x_0\in\mathbb R^n$, $f_0\colon\mathbb R^n\to\mathbb R$ and $g\colon \mathbb R^n\to\mathbb R^m$.

\begin{theorem}\label{T1}
Assume that in problem \eqref{5}--\eqref{7} the set~$U$ is compact,
the set of admissible trajectories is nonempty, and there exists a~constant $K>0$ such that
\begin{equation}\label{8}
|(x,\varphi(t,x,u))|\le K(|x|^2+1)
\end{equation}
for all $t\in[t_0,t_1]$, $x\in\mathbb R^n$ and $u\in U$. Then problem \eqref{5}--\eqref{7}
has a~global infimum.
\end{theorem}

Note that if the conditions of this theorem are augmented with the additional condition
that the set
\eqref{4} be convex for all $t\in[t_0,t_1]$ and $x\in\mathbb R^n$, then we get by Filippov's theorem~\cite{F}
the conditions for the existence of an optimal trajectory for problem \eqref{5}--\eqref{7}.

\begin{proof} \rm
We set $\Sigma^{n+1}=\{\,\ov\alpha=(\alpha_1,\ldots,\alpha_{n+1})\in\mathbb R^{n+1}_+
: \sum_{i=1}^{n+1}\alpha_i=1\,\}$ and define $f(t,x,\ov
u)=\!\sum_{i=1}^{n+1}\alpha_i\varphi(t,x,u_i)$, where $\ov u=\!(u_1,\ldots,
u_{n+1},\alpha_1,\ldots,\alpha_{n+1})\in Q=U^{n+1}\times\Sigma^{n+1}$. Consider the  problem
\begin{align}
& f_0(x(t_1))\to\inf, \label{9}\\[3pt]
&\dot x=f(t,x,\ov u(t)),\quad \ov u(t)\in Q\,\,\,\text{for almost all}\,\,\, t\in[t_0,t_1],\label{10}\\[3pt]
& x(t_0)=x_0,\quad g(x(t_1))=0\label{11}.
\end{align}

It is clear that the set $Q$ is compact and that any admissible trajectory for system
\eqref{6},\,\eqref{7} is admissible also for system  \eqref{10},\,\eqref{11}.
From \eqref{8} and from the form of the function~$f$ one easily finds that
\begin{equation*}
|(x, f(t,x,\ov u))|\le K(|x|^2+1)
\end{equation*}
for all $t\in[t_0,t_1]$, $x\in\mathbb R^n$ and $\ov u\in Q$.

We next note that for each $t\in[t_0,t_1]$ and $x\in\mathbb R^n$ the set
\begin{equation*}
R(t,x)=\{\, f(t,x,\ov u)=\sum_{i=1}^{n+1}\alpha_i\varphi(t,x,u_i)\in\mathbb R^n : \ov u\in Q\,\}
\end{equation*}
is convex, because, on the one hand, $R(t,x)$ is clearly contained in the convex hull of $\varphi(t,x, U)$,
and on the other hand, by Carath\'eodory's theorem, any element from this convex hull
can be represented as a~convex combination of at most $n+1$ elements from $\varphi(t,x, U)$, and
hence, $R(t,x)$ coincides with the convex hull of $\varphi(t,x, U)$.

So, all the hypotheses of A.\,F.~Filippov's theorem are satisfied. This result states in essence that
the set of admissible trajectories for system \eqref{10},\,\eqref{11} is compact. We let~$\gamma$
denote the infimum of the numbers $f_0(x(t_1))$ over all
admissible trajectories~$x\cd$ for system \eqref{6},\,\eqref{7}. These trajectories
are clearly also admissible for system \eqref{10},\,\eqref{11}, and hence $\gamma$~is finite.
There exists a~sequence $x_N\cd$ of these trajectories such that the sequence
$f_0(x_N(t_1))$ converges to~$\gamma$. Hence by compactness it can be assumed that this sequence
converges to some function $\wx\cd\in C([t_0,t_1],\mathbb R^n)$, and therefore,  $\wx\cd$
lies in the closure of the set of all admissible trajectories for system
\eqref{6},\,\eqref{7}. Next, since $f_0$ is continuous, we have
$f_0(\wx(t_1))=\lim_{N\to\infty}f_0(x_N(t_1))=\gamma\le f_0(x(t_1))$ for any admissible
trajectory $x\cd$ for system \eqref{6},\,\eqref{7}. Therefore, $\wx\cd$~is a~global
infimum for problem \eqref{5}--\eqref{7}.
\end{proof}

For any $k\in\mathbb N$, we set
\begin{equation*}
\Sigma^k=\{\,\ov\alpha=(\alpha_1,\ldots,\alpha_k)\in \mathbb R_+^k : \,\,
\sum_{i=1}^k\alpha_i=1\,\}
\end{equation*}
and associate with system \eqref{2},\,\eqref{3} the control system
\begin{align}
&\dot x\! =\!\!\sum_{i=1}^k\alpha_i(t)\varphi(t,x,u_i(t)), \,\,\, \ov u(t)\in  U^k, \,\,
\,\ov\alpha(t)\in \Sigma^k \, \,
\text{a.e. on}
\,\,[t_0,t_1], \label{12} \\[3pt]
&f(x(t_0),x(t_1))\le0,\quad g(x(t_0),x(t_1))=0, \label{13}
\end{align}
where $\ov u\cd=(u_1\cd,\ldots,u_k\cd)$ and $\ov\alpha\cd=(\alpha_1\cd,\ldots,\alpha_k\cd)$.
This system will be called a~{\it convex extension} ({\it relaxation}) of system
\eqref{2},\,\eqref{3} (or simply a~{\it convex system}).

As before, a trajectory $x\cd\in AC([t_0,t_1],\mathbb R^n)$ is called {\it
admissible} for a convex system \eqref{12},\,\eqref{13} if there exist $\ov
u\cd\in (L_\infty([t_0,t_1],\mathbb R^r))^k$ and $\ov\alpha\cd\in (L_\infty([t_0,t_1]))^k$
satisfying conditions \eqref{12},~\eqref{13}.

A triple $(x\cd,\ov u\cd,\ov\alpha\cd)$ will also be called {\it admissible} for
the convex system \eqref{12} \eqref{13}.

We need some more notation.
We let $\la \lambda,x\ra=\sum_{i=i}^n\lambda_ix_i$
denote a~linear functional
$\lambda=(\lambda_1,\ldots,\lambda_n)\in(\mathbb R^n)^*$ evaluated at a~point
$x=(x_1,\ldots,x_n)^T\in\mathbb R^n$ ($T$~is the transpose).
The Euclidean norm of an element $x\in\mathbb R^n$ is denoted by~$|x|$.
By $(\mathbb R^n)^*_+$ we denote the set of functionals on~$\mathbb
R^n$ which assume nonnegative values on nonnegative  vectors.
The adjoint  operator to a~linear operator
$\Lambda\colon \mathbb R^n\to\mathbb R^m$ will be denoted by~$\Lambda^*$.

If $\wx\cd$ is a fixed function, then, to alleviate the notation, the partial derivatives
of the mappings $f_0$, $f$ and~$g$ with respect to $\zeta_1$ and $\zeta_2$ at a~point $(\wx(t_0),\wx(t_1))$
will be written as $\widehat f_{0\zeta_i}$, $\widehat f_{\zeta_i}$ and $\widehat g_{\zeta_i}$,
$i=1,2$.

\begin{theorem}\label{T2}
If a function $\wx\cd\in AC([t_0,t_1],\mathbb R^n)$ is a local infimum for problem
\eqref{1}--\eqref{3}, then, for any $k\in\mathbb N$, $\wou\cd=(\wu_1\cd,\ldots,\wu_k\cd)$
and $\woa\cd=(\widehat\alpha_1\cd,\ldots,\widehat\alpha_k\cd)$, for which triple
$(\wx\cd,\wou\cd,\woa\cd)$ is admissible for the convex system \eqref{12},\,\eqref{13},
there exist a~nonzero tuple $(\lambda_0,\lambda_f,\lambda_g)\in \mathbb R_+\times(\mathbb
R^{m_1})^*_+\times(\mathbb R^{m_2})^*$ and a~vector function $p\cd\in AC([t_0,t_1],(\mathbb
R^n)^*)$ such that the following conditions hold:
\begin{itemize}
\item[$1)$] the stationarity condition  with respect to $x\cd$
$$
\dot p(t) =-p(t)\sum_{i=1}^k\widehat\alpha_i(t)\varphi_x(t,\wx(t),\wu_i(t)),
$$
\item[2)] the transversality condition
$$
p(t_0)=\lambda_0{\widehat {f}}_{0\zeta_1}+{\widehat {f}_{\zeta_1}}^*\lambda_f+{\widehat
{g}_{\zeta_1}}^*\lambda_g,\quad p(t_1)=-\lambda_0{\widehat {f}}_{0\zeta_2}-{\widehat
{f}_{\zeta_2}}^*\lambda_f-{\widehat {g}_{\zeta_2}}^*\lambda_g,
$$
\item[3)] the complementary slackness condition
$$
\la \lambda_f, f(\wx(t_0),\wx(t_1))\ra=0,
$$
\item[4)] the maximum condition for almost all $t\in[t_0,t_1]$
\begin{equation*}
\max_{u\in U}\la p(t),\varphi(t,\wx(t),u)\ra=\la p(t),\dot
{\wx}(t)\ra.
\end{equation*}
\end{itemize}

If \  $U$ is a compact set, then a local infimum for problem \eqref{1}--\eqref{3} is an
admissible trajectory for the convex system \eqref{12},\,\eqref{13} with $k=n+1$.
\end{theorem}

It can be seen that necessary optimality conditions form a~family of relations which is parameterized by triples
$(\wx\cd,\wou\cd,\woa\cd)$,  each relation having the form of a~maximum principle. Moreover,
if $\wx\cd$ is an optimal trajectory for problem
\eqref{1}--\eqref{3}, then these relations contain the classical Pontryagin maximum principle
($k=1$, $u_1\cd=\wu\cd$, $\ov\alpha_1\cd=1$), as well as some other relations, which, in general,
can provide an additional information about the optimal trajectory (see Example~$2$ in the section ``Examples'').
So, the above theorem strengthens  the Pontryagin maximum principle.

If a local infimum is not an admissible trajectory for system
\eqref{2},\,\eqref{3}, then this theorem provides a~tool (similar to a large extent to that based
on the Pontryagin maximum principle for finding the trajectories ``suspicious'' for a~local infimum.
Moreover, if such a~trajectory is found from conditions 1)--4), which are satisfied only for
$\lambda_0\ne0$, then by Theorem~\ref{T3} (to be proved below), this trajectory lies in the closure  of the set
of admissible trajectories for system \eqref{2},\,\eqref{3}.
All this will be illustrated in Example~$1$ in the section ``Examples''.

Example $3$ in the section ``Examples'' shows that the assumption in the last assertion of the theorem
that the set~$U$ be compact is essential.

\smallskip

\noindent \textbf{Proof of Theorem $\ref{T2}$.} \rm
Below, to alleviate the notation, we frequently write $x$, $u$, $\ov u$, $\alpha$,
$\ov\alpha$, etc., in place of $x\cd$, $u\cd$, $\ov u\cd$, $\alpha\cd$, $\ov\alpha\cd$, etc.

A neighborhood of a point $x$ from a normed space will be denoted by~$\wo(x)$.

We introduce the following notation for controls in the convex system \eqref{9},\,\eqref{10}.
We set
\begin{equation*}
\mathcal U=\{\,u\in L_\infty([t_0,t_1],\mathbb R^r) : u(t)\in
U\,\,\,\text{a.~e.\ on}\,\,\,[t_0,t_1]\,\}
\end{equation*}
and define
\begin{equation*}
\mathcal A_k=\{\,\ov\alpha=(\alpha_1,\ldots,\alpha_k)\in (L_\infty([t_0,t_1]))^k :
\ov\alpha(t)\in\Sigma^k \,\,\,\text{a.~e.\ on}\,\,\, [t_0,t_1]\,\},
\end{equation*}
where, we recall, $\Sigma^k=\{\,\ov\alpha=(\alpha_1,\ldots,\alpha_k)\in \mathbb R_+^k : \,\,
\sum_{i=1}^k\alpha_i=1\,\}$.

Let $(\wx,\wou,\woa)$ be a triple from the statement of the theorem
which is admissible  for the convex system \eqref{9},\,\eqref{10}, $N>k$, $\ov\alpha'=(\woa,0,\ldots,0)\in\mathcal A_N$,
$\ov v=(v_1,\ldots,v_{N-k})\in \mathcal U^{N-k}$ and $\ov
u'=(\wu_1,\ldots,\wu_k,v_1,\ldots,v_{N-k})$.

By the condition, $\wx$~is a solution to the differential equation
\begin{equation}\label{gor}
\dot x=\sum_{i=1}^k\widehat\alpha_i(t)\varphi(t,x,\wu_i(t))
\end{equation}
on $[t_0,t_1]$. By Lemma~\ref{L2} there exist neighborhoods $\wo(\wx(t_0))$
and $\wo(\ov\alpha')$ such that, for all $\xi\in \wo(\wx(t_0))$ and
$\ov\alpha=(\alpha_1,\ldots,\alpha_N)\in\wo(\ov\alpha')$, there exists a~unique solution
$x(\cdot,\xi,\ov\alpha;\ov u')$ to the equation
\begin{equation}\label{s1}
\dot x=\sum_{i=1}^N\alpha_i(t)\varphi(t,x,u_i(t)),\quad x(t_0)=\xi,
\end{equation}
on $[t_0,t_1]$, where $u_i=\wu_i$, $i=1,\ldots,k$, $u_{k+i}=v_i$, $i=1,\ldots,
N-k$. Moreover, the mapping $(\xi,\ov\alpha)\mapsto x(\cdot,\xi,\ov\alpha;\ov u')$,
\textit{qua} a~mapping into $C([t_0,t_1],\mathbb R^n)$, is continuously  differentiable.

Let us define the mapping $\widehat\Phi$, which associates with a~quadruple $(\xi,\ov\alpha,\nu_0,\nu)$
from $\wo(\wx(t_0))\times\wo(\ov\alpha')\times\mathbb R\times\mathbb R^{m_1}$  a~vector from
$\mathbb R^{1+m_1+m_2}$ by the rule
\begin{multline}\label{s3}
\widehat\Phi(\xi,\ov\alpha,\nu_0, \nu)=(f_0(\xi,x(t_1,\xi,\ov\alpha;\ov u'))-
f_0(\wx(t_0),\wx(t_1))+\nu_0,\\
f(\xi,x(t_1,\xi,\ov\alpha;\ov u'))+\nu, \ g(\xi,x(t_1,\xi,\ov\alpha;\ov u')))^T.
\end{multline}
The dependence of this mapping on a fixed tuple $\ov u'$ will not be indicated.

The mapping $\widehat\Phi$ is clearly well defined and is continuously differentiable
with respect to $(\xi,\ov\alpha,\nu_0,\nu)$ by the properties of the mapping $(\xi,\ov\alpha)\mapsto
x(\cdot,\xi,\ov\alpha;\ov u')$ and the mappings $f_0$, $f$ and~$g$.

The scheme of the proof of the necessary conditions in Theorem~\ref{T2} is as follows.  We show that the
inclusion
\begin{equation}\label{s4}
0\in{\rm int}\,\widehat\Phi'(\ww) (\mathbb R^n\times (\mathcal
A_N-\ov\alpha')\times\mathbb R_+\times(\mathbb R^{m_1}_++f(\wx(t_0),\wx(t_1)))),
\end{equation}
where $\ww=(\wx(t_0),\ov\alpha',0,-f(\wx(t_0),\wx(t_1)))$, contradicts the fact that $\wx$~is a~local infimum
for problem \eqref{1}--\eqref{3}. Then, separating from zero the convex set on the right of~\eqref{s4}
for each $N>k$ and
$\ov v=(v_1,\ldots,v_{N-k})\in \mathcal U^{N-k}$, we get  all the necessary conditions formulated in the theorem.

Let us use Lemma \ref{L1}  in the case $X=\mathbb
R^n\times(L_\infty([t_0,t_1]))^N\times\mathbb R\times \mathbb R^{m_1}$, \ $K=\mathbb
R^n\times \mathcal A_N\times\mathbb R_+\times\mathbb R^{m_1}_+$, and
$\ww=(\wx(t_0),\ov\alpha',0,-f(\wx(t_0),\wx(t_1)))$.

Let $\wo_0(\wx(t_0))$ and $\wo_0(\ov\alpha')$ be neighborhoods from Lemma \ref{L4}. Reducing these neighborhoods
and considering bounded neighborhoods $\wo_0(0)$ (of the origin in~$\mathbb R$) and
$\wo_0(-f(\wx(t_0),\wx(t_1))$, one can assume that the mapping $\widehat \Phi$ is bounded on
$V=\wo_0(\wx(t_0))\times\wo_0(\ov\alpha')\times\wo_0(0)\times\wo_0(-f(\wx(t_0),\wx(t_1))$.

The inclusion  \eqref{s4} is equivalent to the inclusion $0\in{\rm int}\,\widehat\Phi'(\ww)(K-\ww)$. So,
all the hypotheses of Lemma~\ref{L1} are satisfied.

By Lemma~\ref{4}, for sufficiently large $s\in\mathbb N$, the continuous
mappings $(\xi,\ov\alpha)\mapsto x_s(\cdot,\xi,\ov\alpha;\ov u')$ from $\mathcal
M=\wo_0(\wx(t_0))\times(\wo_0(\ov\alpha')\cap \mathcal A_N)$ into $C([t_0,t_1],\mathbb
R^n)$ are defined. Hence, for such~$s$, the
continuous mappings
\begin{multline*}
\Phi_s(\xi,\ov\alpha,\nu_0,\nu)=(f_0(\xi, x_s(t_1,\xi,\ov\alpha;\ov
u'))-f_0(\wx(t_0),\wx(t_1))
+\nu_0,\\
f(\xi, x_s(t_1,\xi,\ov\alpha;\ov u'))+\nu, \ g(\xi, x_s(t_1,\xi,\ov\alpha;\ov u')))^T
\end{multline*}
are defined on $\mathcal M\times\mathbb R\times\mathbb R^{m_1}$.
Since the mappings $(\xi,\ov\alpha)\mapsto x_s(\cdot,\xi,\ov\alpha;\ov u')$
lie in the space $C(\mathcal M,\, C([t_0,t_1],\,\mathbb R^n))$ and converge in this space
to the mapping $(\xi,\ov\alpha)\mapsto x(\cdot,\xi,\ov\alpha;\ov u')$ as
$s\to\infty$, and since the mappings
$f_0$, $f$ and~$g$ are continuously differentiable, it easily follows that the mappings $(\xi,\ov\alpha,\nu_0,\nu)\mapsto
\Phi_s(\xi,\ov\alpha,\nu_0,\nu)$ lie in the space $C(V\cap K,\,\mathbb
R^{1+m_1+m_2})$ (with reduced neighborhood~$V$, if necessary) and converge in this space
to the mapping $(\xi,\ov\alpha,\nu_0,\nu)\mapsto \widehat\Phi(\xi,\ov\alpha,\nu_0,\nu)$ as $s\to\infty$.

Let $\varepsilon>0$. There exists $s_0\in\mathbb N$ such that
$\|x_s(\cdot,\xi,\ov\alpha;\ov
u')-\wx\cd\|_{C([t_0,t_1],\mathbb R^n)}<\varepsilon/2$ for all $s\ge s_0$ and
$(\xi,\ov\alpha)\in\mathcal M$. Next, since the mapping
$(\xi,\ov\alpha)\mapsto x(\cdot,\xi,\ov\alpha;\ov u')$ is continuous at
$(\wx(t_0),\ov\alpha')$, there exists $\delta_0>0$ such that
$\|x(\cdot,\xi,\ov\alpha;\ov u')-\wx\cd\|_{C([t_0,t_1],\mathbb R^n)}<\varepsilon/2$ if
$|\xi-\wx(t_0)|+\|\ov\alpha-\ov\alpha'\|_{(L_\infty([t_0,t_1]))^N}<\delta_0$.

As a result, we see that if $s\ge s_0$ and if a pair $(\xi,\ov\alpha)\in\mathcal M$ is such that
$|\xi-\wx(t_0)|+\|\ov\alpha-\ov\alpha'\|_{(L_\infty([t_0,t_1]))^N}<\delta_0$, then
\begin{multline}\label{***}
\|x_s(\cdot,\xi,\ov\alpha;\ov u')-\wx\cd\|_{C([t_0,t_1],\mathbb
R^n)}\le\|x_s(\cdot,\xi,\ov\alpha;\ov u')\\-x(\cdot,\xi,\ov\alpha;\ov
u')\|_{C([t_0,t_1],\mathbb R^n)}+ \|x(\cdot,\xi,\ov\alpha;\ov
u')-\wx\cd\|_{C([t_0,t_1],\mathbb R^n)}<\varepsilon.
\end{multline}

Let a neighborhood $V_0\subset V$ of the point $\ww$ and constants $r_0$ and $\gamma$ be from Lemma~\ref{L1}.
We choose $r\in(0,r_0]$ so as to have $\gamma r\le\delta_0$ and let $s\ge s_0$
be such that $\Phi_s\in U_{C(V\cap K,\,\mathbb R^{1+m_1+m_2})}(\widehat \Phi,r)$.

We have $\widehat\Phi(\ww)=0$, and hence the pairs  $(\ww,y)$, where $y\in \mathbb R^{1+m_1+m_2}$ and
$|y|\le r$, satisfy relation \eqref{5n} of Lemma~\ref{L1}.

Let $z=(\ww,y)$ be such a pair and let  $y$ be of the form $y=(y_1,0)$, where $y_1<0$.
If $g_{\Phi_s}$ is the mapping from this lemma, then for this pair the lemma asserts that
(here we denote $g_{\Phi_s}(z)=w_z=(\xi_z,\ov\alpha_z,\nu_{0z},\nu_z)$)
\begin{equation}\label{exa3}
\begin{aligned}
f_0(\xi_z, x_s(t_1,\xi_z,\ov\alpha_z;\ov u'))-f_0(\wx(t_0),\wx(t_1))+\nu_{0z}&=y_1,\\
f(\xi_z, x_s(t_1,\xi_z,\ov\alpha_z;\ov u'))+\nu_z&=0,\\
g(\xi_z, x_s(t_1,\xi_z,\ov\alpha_z;\ov u'))&=0
\end{aligned}
\end{equation}
and
\begin{equation}\label{exa4}
\|w_z-\ww\|_X\le\gamma r.
\end{equation}

By Lemma~\ref{L4} the function $x_s(\cdot,\xi_z,\ov\alpha_z;\ov u')$  is a~solution
of the equation
\begin{equation*}
\dot x=\varphi(t,x,u_s(\ov\alpha_z;\ov u')(t)),\quad x(t_0)=\xi_z.
\end{equation*}
From the definition of $u_s(\ov\alpha_z;\ov u')$ (see Lemma~\ref{L3}) it follows that
$u_s(\ov\alpha_z;\ov u')(t)\in U$ for almost all $t\in[t_0,t_1]$.  It is also clear that
$\xi_z=x_s(t_0,\xi_z,\ov\alpha_z;\ov u')$.

Now from the second and third relations in \eqref{exa3} and using the inequality
$\nu_z\ge0$, it follows that the function $x_s(\cdot,\xi_z,\ov\alpha_z;\ov u')$ is admissible
for system \eqref{2},\,\eqref{3}. Moreover, from the first relation it follows that on this function the value of the functional $f_0$ is not smaller
than on~$\wx$ ($\nu_{0z}\ge0$, $y_1<0$).

Next, from \eqref{exa4} and the choice of $r$,
\begin{equation*}
|\xi_z-\wx(t_0)|+\|\ov\alpha_z-\ov\alpha'\|_{(L_\infty([t_0,t_1]))^N}\le\|w_z-\ww\|_X\le\gamma
r\le\delta_0,
\end{equation*}
and hence by \eqref{***} we have $\|x_s(\cdot,\xi_z,\ov\alpha_z;\ov
u')-\wx\cd\|_{C([t_0,t_1],\mathbb R^n)}<\varepsilon$.

So, in any neighborhood of the point~$\wx$ there exists a~function which is
admissible  for system \eqref{2},\,\eqref{3} and on which the value of the functional to be minimized
is smaller than on~$\wx$. This contradicts the fact that  $\wx$~is a~local infimum for problem \eqref{1}--\eqref{3}.

So, inclusion \eqref{s4} does not hold for any $N>k$ and any tuple $\ov v=(v_1,\ldots,v_{N-k})\in\mathcal U^{N-k}$.
Therefore, for any such $N$ and~$\ov v$, it follows from the separation theorem that there exists
a~nonzero vector $\lambda(\ov v)\in (\mathbb R^{1+m_1+m_2})^*$ for which
\begin{equation}\label{maga}
\la\lambda(\ov
v),\Phi'(\ww)[\xi,\,\ov\alpha-\ov\alpha',\,\nu_0,\,\nu+f(\wx(t_0),\wx(t_1))]\ra\ge0
\end{equation}
for all $(\xi,\ov\alpha,\nu_0,\nu)\in\mathbb R^n\times \mathcal A_N\times\mathbb
R_+\times\mathbb R^{m_1}_+$.

Let $\lambda(\ov v)=(\lambda_0(\ov v),\lambda_1(\ov
v),\lambda_{2}(\ov v))\in \mathbb R\times(\mathbb R^{m_1})^*\times
(\mathbb R^{m_2})^*$. By the chain rule, inequality \eqref{maga} can be written as
\begin{multline}\label{maga1}
\lambda_0(\ov v)(\la\widehat f_{0\zeta_1},\,\xi\ra+\la \widehat
f_{0\zeta_2},\, (\wx_\xi\xi)(t_1)\ra+\la \widehat
f_{0\zeta_2},\,(\widehat
x_{\ov\alpha}(\ov\alpha-\ov\alpha'))(t_1)\ra+\nu_0)\\+\la\lambda_1(\ov
v),\,\widehat f_{\zeta_1}\xi+\widehat
f_{\zeta_2}(\wx_\xi\xi)(t_1)+\widehat
f_{\zeta_2}(\wx_{\ov\alpha}(\ov\alpha-\ov\alpha'))(t_1)
+\nu+f(\wx(t_0),\wx(t_1))\ra\\+\la\lambda_2(\ov v),\,\widehat
g_{\zeta_1}\xi+\widehat g_{\zeta_2}(\wx_\xi\xi)(t_1)+\widehat
g_{\zeta_2}(\wx_{\ov\alpha}(\ov\alpha-\ov\alpha'))(t_1)\ra\ge0
\end{multline}
for any $(\xi,\ov\alpha,\nu_0,\nu)\in\mathbb R^n\times \mathcal
A_N\times\mathbb R_+\times\mathbb R^{m_1}_+$, where
$\wx_\xi=\wx_\xi(\ov v)$ and $\wx_{\ov\alpha}=\wx_{\ov\alpha}(\ov v)$
are the partial derivatives of the mapping
$(\xi,\ov\alpha)\mapsto x(\cdot,\xi,\ov\alpha;\ov u')$ at the point $(\wx(t_0),\ov\alpha')$
with respect to $\xi$ and~$\ov\alpha$, respectively.

Let us show that there exists a~tuple
$\lambda=(\lambda_0,\lambda_f,\lambda_g)\in \mathbb R\times(\mathbb
R^{m_1})^*\times (\mathbb R^{m_2})^*$, $|\lambda|=1$, such that
\eqref{maga1} holds with this~$\lambda$ for any
$N>k$ and any tuple $\ov v=(v_1,\ldots,v_{N-k})\in\mathcal U^{N-k}$.

For a given tuple $\ov v=(v_1,\ldots,v_{N-k})\in\mathcal U^{N-k}$
we denote by $\Lambda_N(\ov v)$ the set of all such vectors
$\lambda(\ov v)$, $|\lambda(\ov v)|=1$, that satisfy
\eqref{maga1}. It is clear that $\Lambda_N(\ov v)$ is a~closed
subset of the (compact) unit sphere of $(\mathbb R^{1+m_1+m_2})^*$. Let us check that the family
$\mathcal A$ of all such subsets (over all $N>k$ and tuples $\ov
v=(v_1,\ldots,v_{N-k})\in\mathcal U^{N-k}$) has the finite intersection property.

Let $\Lambda_{N_j}(\ov v^j)$, $j=1,\ldots,s$, be an arbitrary family of sets
from~$\mathcal A$ ($N_j-k$~is the length of the vector~$\ov
v^j$). Let us show that $\cap_{j=1}^s\Lambda_{N_j}(\ov v^j)\ne\emptyset$.

Indeed, we set $\ov v=(\ov v^1,\ldots,\ov v^s)$ (the length of~$\ov v$
is denoted by $N$). Let $1\le j\le s$.
Consider the ffamily $\ov v^j$ and define $\ov \alpha^j
=(\alpha_1,\ldots,\alpha_k,\alpha_{k+1},\ldots,\alpha_{N_j})\in
\mathcal A_{N_j}$.

We augment the vector $\ov \alpha^j $ by zero functions to the vector $\ov
\alpha$ of length~$N$. It is clear that $\ov \alpha\in \mathcal A_N$.

Setting $\ov \alpha'_{N_j}=(\woa,0)\in\mathcal A_{N_j}$, it is easily seen that
\begin{equation*}
(\wx_{\ov\alpha}(\ov\alpha-\ov\alpha'))(t_1)=
x_{\ov\alpha^j}(\wx(t_0),\ov
\alpha'_{N_j})[\ov\alpha^j-\ov\alpha'_{N_j}](t_1).
\end{equation*}
Hence, using \eqref{maga1}, we get the inclusion $\lambda(\ov
v)\in\Lambda_{N_j}(\ov v^j)$, thereby showing that $\lambda(\ov v)\in \cap_{j=1}^s\Lambda_{N_j}(\ov v^j)$.

So, the system of sets~$\mathcal A$ has the finite intersection property, and hence,
there exist $\lambda_0\in\mathbb R$, $\lambda_f\in (\mathbb
R^{m_1})^*$ and $\lambda_g\in (\mathbb R^{m_2})^*$ such that \eqref{maga1}
holds for any tuple~$\ov v$. In particular,
this relation holds for the tuples consisting of a~single element $\ov v=v$; that is,
\begin{multline}\label{maga3}
\lambda_0(\la\widehat f_{0\zeta_1},\,\xi\ra+\la\widehat
f_{0\zeta_2},\,(\wx_\xi(v)\xi)(t_1)+\la\widehat
f_{0\zeta_2},\,(\widehat
x_{\ov\alpha}(v)(\ov\alpha-\ov\alpha'))(t_1)\ra+\nu_0)\\+\la\lambda_f,\,\widehat
f_{\zeta_1}\xi+\widehat f_{\zeta_2}(\wx_\xi(v)\xi)(t_1)+\widehat
f_{\zeta_2}(\wx_{\ov\alpha}(v)(\ov\alpha-\ov\alpha')(t_1)
+\nu+f(\wx(t_0),\wx(t_1))\ra\\+\la\lambda_g,\,\widehat
g_{\zeta_1}\xi+\widehat g_{\zeta_2}(\wx_\xi(v)\xi)(t_1)+\widehat
g_{\zeta_2}(\wx_{\ov\alpha}(v)(\ov\alpha-\ov\alpha'))(t_1)\ra\ge0
\end{multline}
for all $(\xi,\ov\alpha,\nu_0,\nu)\in\mathbb R^n\times \mathcal
A_{k+1}\times\mathbb R_+\times\mathbb R^{m_1}_+$ and $v\in \mathcal
U$.

Now we employ this inequality to derive the necessary conditions from the theorem.

Setting $\xi=0$, $\ov\alpha=\alpha'$ and
$\nu=-f(\wx(t_0),\wx(t_1))$ in \eqref{maga3}, we see that $\lambda_0\nu_0\ge0$ for any $\nu_0\ge0$,
and hence, $\lambda_0\ge0$.

If $\xi=0$, $\ov\alpha=\ov\alpha'$, $\nu_0=0$ and $\nu=\nu'-f(\wx(t_0),\wx(t_1))$, where $\nu'\in \mathbb R^{m_1}_+$, then
from \eqref{maga3} it follows that $\la\lambda_f,\,\nu'\ra\ge0$ for any $\nu'\in\mathbb R^{m_1}_+$; that is, $\lambda_f\in(\mathbb
R^{m_1})^*_+$.

Let $\xi=0$, $\ov\alpha=\ov\alpha'$, $\nu_0=0$ and $\nu=0$. Now from inequality \eqref{maga3}
it follows that $\la\lambda_f,\,f(\wx(t_0),\wx(t_1))\ra\ge0$. But $\lambda_f\in(\mathbb
R^{m_1})^*_+$, $f(\wx(t_0),\wx(t_1))\le0$, and therefore,
$\la\lambda_f,\,f(\wx(t_0),\wx(t_1))\ra\le0$; that is,
$\la\lambda_f,\,f(\wx(t_0),\wx(t_1))\ra=0$, which proves the complementary slackness condition.

Let $p$ be the solution of the equation
\begin{equation}\label{maga4}
\begin{gathered}
\dot p
=-p\sum_{i=1}^k\widehat\alpha_i(t)\varphi_x(t,\wx(t),\wu_i(t)),\\p(t_1)=-\lambda_0{\widehat
{f}}_{0\zeta_2}-{\widehat {f}_{\zeta_2}}^*\lambda_f-{\widehat
{g}_{\zeta_2}}^*\lambda_g.
\end{gathered}
\end{equation}
In \eqref{maga3} we put $\ov\alpha=\ov\alpha'$, $\nu_0=0$ and $\nu=-f(\wx(t_0),\wx(t_1))$.
Since $\xi\in\mathbb R^n$, we have
\begin{multline}\label{maga3n}
\lambda_0(\la \widehat f_{0\zeta_1},\,\xi\ra+\la \widehat
f_{0\zeta_2},\,(\wx_\xi(v)\xi)(t_1)\ra)+\la \lambda_f,\,\widehat
f_{\zeta_1} \xi+\widehat f_{\zeta_2} (\wx_\xi(v)\xi)(t_1)\ra\\+\la
\lambda_g,\,\widehat g_{\zeta_1} \xi+\widehat g_{\zeta_2}
(\wx_\xi(v)\xi)(t_1)\ra=0.
\end{multline}
From \eqref{s2} of Lemma~\ref{L2} it follows that the derivative $\wx_\xi(v)$ (which is identified
with the corresponding matrix function) satisfies the equation
\begin{equation}\label{maga3nn}
\dot
\wx_\xi(v)=\sum_{i=1}^k\widehat\alpha_i(t)\varphi_x(t,\wx(t),\wu_i(t))\wx_\xi(v),\quad
\wx_\xi(v)(t_0)=E,
\end{equation}
where $E$ is the identity matrix.

From \eqref{maga3n}, \eqref{maga4} and \eqref{maga3nn}  we see that
\begin{multline*}
\la\lambda_0\widehat f_{0\zeta_1}+\widehat
f_{\zeta_1}^*\lambda_f+\widehat g_{\zeta_1}^*\lambda_g,\,\xi\ra
=-\la\lambda_0\widehat f_{0\zeta_2}+\widehat
f_{\zeta_2}^*\lambda_f+\widehat g_{\zeta_2}^*\lambda_g,\,
(\wx_\xi(v)\xi)(t_1)\ra\\=\la
p(t_1),\,(\wx_\xi(v)\xi)(t_1)\ra=\int_{t_0}^{t_1}(\la
p(t),\,\frac{d}{dt}(\wx_\xi(v)\xi)(t)\ra\\+\la\dot
p(t),\,(\wx_\xi(v)\xi)(t)\ra)\,dt+\la
p(t_0),\,(\wx_\xi(v)\xi)(t_0)\ra\\= \la
p(t_0),\,(\wx_\xi(v)\xi)(t_0)\ra=\la p(t_0),\,\xi\ra
\end{multline*}
and hence,
\begin{equation*}
p(t_0)=\lambda_0\widehat f_{0\zeta_1}+\widehat
f_{\zeta_1}^*\lambda_f+\widehat g_{\zeta_1}^*\lambda_g.
\end{equation*}
This together with \eqref{maga4} proves assertions $1)$ and $2)$ of the theorem. Let us now prove the maximum condition.

In \eqref{maga3} we put $\xi=0$, $\nu_0=0$ and $\nu=-f(\wx(t_0),\wx(t_1))$. Hence
\begin{equation}\label{osn}
\la\lambda_0\widehat f_{0\zeta_2}+\widehat
f_{\zeta_2}^*\lambda_f+\widehat g_{\zeta_2}^*\lambda_g, \
(\wx_{\ov\alpha}(v)(\ov\alpha-\ov\alpha')(t_1)\ra\ge0
\end{equation}
for all $\ov\alpha\in \mathcal A_{k+1}$ and $v\in \mathcal U$.

Another appeal to \eqref{s2} shows that for any
 $\ov\alpha=(\alpha_1,\ldots,\alpha_{k+1})\in(L_\infty([t_0,t_1]))^{k+1}$
the function $q(v,\ov\alpha)=\wx_{\ov\alpha}(v)\ov\alpha$  satisfies the equation
\begin{multline}\label{maga4n}
\dot q(v,\ov\alpha)
=\sum_{i=1}^k\widehat\alpha_i(t)\varphi_x(t,\wx(t),\wu_i(t))q(v,\ov\alpha)+
\sum_{i=1}^{k}\alpha_i(t)\varphi(t,\wx(t),\wu_i(t))
\\+\alpha_{k+1}(t)\varphi(t,\wx(t),v(t)),\quad
q(v,\ov\alpha)(t_0)=0.
\end{multline}
Let $\ov\alpha\in \mathcal A_{k+1}$. From \eqref{maga4}, \eqref{osn}
and \eqref{maga4n}  we have
{\allowdisplaybreaks
\begin{multline*}
0\le-\la p(t_1),\,
q(v,\ov\alpha-\ov\alpha')(t_1)\ra\\=\int_{t_0}^{t_1}(\la p(t),\,\dot
q(v,\ov\alpha-\ov\alpha')(t)\ra+\la \dot p(t),\,
q(v,\ov\alpha-\ov\alpha')(t)\ra)dt\\
=-\int_{t_0}^{t_1}(\la
p(t),\,\sum_{i=1}^k\widehat\alpha_i(t)\varphi_x(t,\wx(t),\wu_i(t))
q(v,\ov\alpha-\ov\alpha')(t)
\\+\sum_{i=1}^{k}(\alpha_i(t)-\widehat\alpha_i(t))\varphi(t,\wx(t),\wu_i(t))
+\alpha_{k+1}(t)\varphi(t,\wx(t),v(t))\ra\\-\la
p(t)\sum_{i=1}^k\widehat\alpha_i(t)\varphi_x(t,\wx(t),
\wu_i(t)),\,q(v,\ov\alpha-\ov\alpha')(t)\ra)\,dt.
\end{multline*}
}
It follows that
\begin{multline*}
\int_{t_0}^{t_1}\alpha_{k+1}(t)\la p(t),\,\varphi(t,\wx(t),v(t))\ra\,dt\\\le
\int_{t_0}^{t_1}\la
p(t),\,\sum_{i=1}^{k}(\widehat\alpha_i(t)-\alpha_i(t))\varphi(t,\wx(t),\wu_i(t))\ra\,dt
\end{multline*}
for any tuple $\ov\alpha=(\alpha_1,\ldots,\alpha_{k+1})\in \mathcal A_{k+1}$ and $v\in \mathcal U$.

For any $1\le i\le k$, we set
$(\widehat\alpha_1,\ldots,\widehat\alpha_{i-1},(1/2)\widehat\alpha_{i},\widehat\alpha_{i+1},\ldots,
\widehat\alpha_k,(1/2)\widehat\alpha_{i})\in \mathcal A_{k+1}$. Substituting this tuple into the last inequality,
we see that
\begin{multline}\label{maga5n}
\int_{t_0}^{t_1}\widehat\alpha_{i}(t)\la p(t),\,\varphi(t,\wx(t),v(t))\ra\,dt\\\le
\int_{t_0}^{t_1}\widehat\alpha_i(t)\la p(t),\,\varphi(t,\wx(t),\wu_i(t))\ra\,dt,\quad
i=1,\ldots,k,
\end{multline}
for all $v\cd\in \mathcal U$.

We let $T_0$ denote the set of Lebesgue points of the functions $\widehat\alpha_i\cd$ and $\la
p\cd,\,\varphi(\cdot,\wx\cd,\wu_i)\ra$, $i=1,\ldots,k$, on $(t_0,t_1)$. Since these
functions are essentially bounded, it can be easily checked that $T_0$ is the set of Lebesgue points
also for the functions  $\widehat\alpha_i\cd\la p\cd,\,\varphi(\cdot,\wx\cd,\wu_i)\ra$, $i=1,\ldots,k$.

Let $\tau\in T_0$. We fix $1\le i\le k$. For any $h>0$ such that $[\tau-h,\tau+h]\subset (t_0,t_1)$, we set
$v_h(t)=v$ if $t\in [\tau-h,\tau+h]$ and $v_h(t)=\wu_i(t)$ if $t\in
[t_0,t_1]\setminus[\tau-h,\tau+h]$. It is clear that $v_h\cd\in\mathcal U$ and so, using \eqref{maga5n},
\begin{equation*}
\frac1{2h}\int_{\tau-h}^{\tau+h}\widehat\alpha_{i}(t)\la
p(t),\,\varphi(t,\wx(t),v)\ra\,dt\le
\frac1{2h}\int_{\tau-h}^{\tau+h}\widehat\alpha_i(t)\la
p(t),\,\varphi(t,\wx(t),\wu_i(t))\ra\,dt.
\end{equation*}
Since the function $\varphi(\cdot,\wx,v)$ is continuous, $\tau$ is its Lebesgue point, and moreover, by the above,
$\tau$~is also a~Lebesgue point for the function  $\widehat\alpha_i\cd\la
p\cd,\,\varphi(\cdot,\wx,v)\ra$. Making $h\to0$ in the last inequality, we find that
\begin{equation*}
\widehat\alpha_i(\tau)\la p(\tau),\,\varphi(\tau,\wx(\tau),v)\ra
\le\widehat\alpha_i(\tau)\la p(\tau),\,\varphi(\tau,\wx(\tau),\wu_i(\tau))\ra
\end{equation*}
for each  $i=1,\ldots,k$.

Adding these inequalities and taking into account equation \eqref{gor}, which
becomes a~sharp equality
at the Lebesgue points of the function on the right, we get the relation
\begin{equation*}
\la p(\tau),\,\varphi(\tau,\wx(\tau),v)\ra \le\la
p(\tau),\,\sum_{i=1}^k\widehat\alpha_i(\tau)\varphi(\tau,\wx(\tau),\wu_i(\tau))\ra =\la
p(\tau),\,\dot{\wx}(\tau)\ra.
\end{equation*}
Since $v\in U$ is arbitrary and since $T_0$ is a~set of full measure, this relation
is equivalent to condition~$4)$ of the theorem. So, all the necessary conditions in the theorem are proved.

Let us prove the last assertion of the theorem. By definition of local infimum, there exists
a~sequence of admissible trajectories for system  \eqref{2},\,\eqref{3}, which converges to~$\wx\cd$.
It is clear that these trajectories are also admissible for the convex system
\eqref{12},\,\eqref{13} for any~$k$. To prove that the function $\wx\cd$
is also admissible for this system with $k=n+1$ we employ Filippov's theorem from
(\cite{F}), in which it is shown, in particular (in our setting) that
if $Q$~is a~compact set and the set
\begin{equation*}
R(t,x)=\{\, \sum_{i=1}^{n+1}\alpha_i\varphi(t,x,u_i)\in\mathbb R^n :
(u_1,\ldots,u_{n+1},\alpha_1,\ldots,\alpha_{n+1})\in Q\,\}
\end{equation*}
is convex for any $t\in[t_0,t_1]$ and $x\in\mathbb R^n$, then the limit of a~converging
sequence of admissible trajectories is also an admissible trajectory.
Since in our setting the set~$Q$ is clearly
compact, and since the convexity of the set $R(t,x)$ for any
$t\in[t_0,t_1]$ and $x\in\mathbb R^n$ is secured by Theorem~\ref{T1}, this proves the
last assertion of Theorem~\ref{T2}.

\smallskip

To formulate the next result we introduce the concept of regularity of the convex system
\eqref{12},\,\eqref{13}.

Let $k\in\mathbb N$ and let $(\wx\cd,\wou\cd,\woa\cd)$ be an admissible triple for the convex system
\eqref{12},\,\eqref{13}. By $\Lambda(\wx\cd,\wou\cd,\woa\cd)$ we denote the set
of tuples $(\lambda_f,\lambda_g,p\cd)\in (\mathbb R^{m_1})^*\times(\mathbb R^{m_2})^*\times
AC([t_0,t_1],(\mathbb R^n)^*)$, where $\lambda_f$ and $\lambda_g$ are not simultaneously zero, satisfying
the relations
\begin{equation}\label{reg}
\begin{aligned}
&\dot p(t) =-p(t)\sum_{i=1}^k\widehat\alpha_i(t)\varphi_x(t,\wx(t),\wu_i(t)),\\
&p(t_0)={\widehat {f}_{\zeta_1}}^*\lambda_f+{\widehat
{g}_{\zeta_1}}^*\lambda_g,\quad p(t_1)=-{\widehat
{f}_{\zeta_2}}^*\lambda_f-{\widehat {g}_{\zeta_2}}^*\lambda_g,\\[5pt]
&\la \lambda_f, f(\wx(t_0),\wx(t_1))\ra=0,\\[3pt]
&\max_{u\in U}\la p(t),\varphi(t,\wx(t),u)\ra=\la p(t),\dot
{\wx}(t)\ra\,\,\,\text{for almost all}\,\,\, t\in[t_0,t_1].
\end{aligned}
\end{equation}
The condition $\Lambda(\wx\cd,\wou\cd,\woa\cd)\ne\emptyset$ means that
the necessary conditions of geometric optimality in the form  of a~maximum principle are satisfied for the
convex system
\eqref{12},\,\eqref{13} at the point $(\wx\cd,\wou\cd,\woa\cd)$. Therefore, the negation of this condition
(that is, the case when $\Lambda(\wx\cd,\wou\cd,\woa\cd)=\emptyset$) can be looked upon as a~regularity
condition for the convex system \eqref{12},\,\eqref{13} at the point
$(\wx\cd,\wou\cd,\woa\cd)$. Taking this into account, we say that the convex system \eqref{12},\,\eqref{13} is {\it
regular at a~point $(\wx\cd,\wou\cd,\woa\cd)$} if $\Lambda(\wx\cd,\wou\cd,\woa\cd)=\emptyset$.

\begin{theorem}\label{T3}
If a convex system \eqref{12},\,\eqref{13} is regular at a~point
$(\wx\cd,\wou\cd,\woa\cd)$, then $\wx\cd$ lies in the closure of the set of admissible trajectories
for system \eqref{2},\,\eqref{3}.
\end{theorem}

\begin{proof} \rm
At the beginning of the proof of the previous theorem, we introduced the mapping $\widehat\Phi$.
Let us consider here its ``truncated'' variant
\begin{equation*}
\widetilde\Phi(\xi,\ov\alpha, \nu)=(f(\xi,x(t_1,\xi,\ov\alpha;\ov u'))+\nu, \
g(\xi,x(t_1,\xi,\ov\alpha;\ov u')))^T,
\end{equation*}
which differs from the mapping $\widehat\Phi$ by the absence of the component
$f_0(\xi,x(t_1,\xi,\ov\alpha;\ov u'))- f_0(\wx(t_0),\wx(t_1))+\nu_0$.

We claim that if a convex system \eqref{12},\,\eqref{13} is regular at a~point
$(\wx\cd,\wou\cd,\woa\cd)$, then, for some $N>k$ and a~tuple $\ov
v=(v_1,\ldots,v_{N-k})\in\mathcal U^{N-k}$,
\begin{equation}\label{kk}
0\in{\rm int}\,\widetilde\Phi'(\widetilde w) (\mathbb R^n\times (\mathcal
A_N-\ov\alpha')\times(\mathbb R^{m_1}_++f(\wx(t_0),\wx(t_1)))),
\end{equation}
where $\widetilde w=(\wx(t_0),\ov\alpha',-f(\wx(t_0),\wx(t_1)))$ (inclusion \eqref{kk} is similar to
inclusion~\eqref{s4}).

Indeed, if inclusion \eqref{kk} is not satisfied for any $N>k$ and any tuple $\ov
v=(v_1,\ldots,v_{N-k})\in\mathcal U^{N-k}$, then arguing as in Theorem~\ref{T2}, we conclude that
conditions \eqref{reg} (which coincide with the necessary conditions in Theorem~\ref{T2} with $\lambda_0=0$)
hold with some nonzero tuple $(\lambda_f,\lambda_g)$, contradicting the assumption.

But if inclusion \eqref{kk} holds, then arguing again as in the proof of Theorem~\ref{T2}
in the part pertaining to the inverse function lemma we conclude that, for each
$\varepsilon>0$, there exists an admissible function for system \eqref{2},\,\eqref{3}  which differs
by lesser than~$\varepsilon$ from $\wx\cd$ in the metric of $C([t_0,t_1],\mathbb R^n)$, which proves the theorem.
\end{proof}

Let us derive two corollaries from this theorem. Let consider the problem of minimization of the functional $f_0$
(see~\eqref{1}) on trajectories of the convex system \eqref{12},\,\eqref{13}. This problem will be referred to as
the convex problem \eqref{1},\,\eqref{12},\,\eqref{13}. The concept of an
optimal trajectory for this problem is defined in a~natural way. Moreover, when speaking
about an admissible trajectory for problem \eqref{1},\,\eqref{12},\,\eqref{13} we imply that this trajectory
is admissible for system \eqref{12},\,\eqref{13}, which imposes constraints in this problem.

\begin{corollary}\label{co1}
If a convex system \eqref{12},\,\eqref{13} is regular at a~point point
$(\wx\cd,\wou\cd,\woa\cd)$ and if $\wx\cd$ is an optimal trajectory in the convex problem
\eqref{1},\,\eqref{12},\,\eqref{13}, then $\wx\cd$~is a~local infimum for problem
\eqref{1}--\eqref{3}.
\end{corollary}

\begin{proof} \rm
By the hypothesis, there exists a~neighborhood $V$ of the point $\wx\cd$ such that $f_0(x(t_0),x(t_1))\ge
f_0(\wx(t_0),\wx(t_1))$ for any
trajectory $x\cd\in V$ admissible for system \eqref{12},\,\eqref{13}. In particular, this is true if $x\cd$ is an admissible trajectory for
system \eqref{2},\,\eqref{3}. Since by Theorem~\ref{T3} the function $\wx\cd$
lies in the closure of the set of admissible trajectories for system \eqref{2},\,\eqref{3},
it follows that $\wx\cd$~is a~local infimum for problem \eqref{1}--\eqref{3}.
\end{proof}

We now give the definition of a~sliding regime.

\begin{definition}\label{d4}\rm
By a \textit{sliding regime} for system \eqref{2},\,\eqref{3} we mean a~function lying in the closure
of the set of admissible trajectories for this system but which does not lie in this set.
\end{definition}

It is clear that if under the hypotheses of Corollary  \ref{co1} an optimal trajectory $\wx\cd$ in the convex
problem \eqref{1},\,\eqref{12},\,\eqref{13} is not an admissible trajectory for system
\eqref{2},\,\eqref{3}, then $\wx\cd$ is a~sliding regime for this system.

\begin{corollary}\label{sl2}
If a convex system \eqref{12},\,\eqref{13} is regular at a~point
$(\wx\cd,\wou\cd,\woa\cd)$ and if $\wx\cd$ is not an admissible trajectory for system
\eqref{2},\,\eqref{3}, then $\wx\cd$~is a~sliding regime for this system. On the other hand,
if $\wx\cd$~is a~sliding regime for system \eqref{2},\,\eqref{3} and if the set
$U$~is compact, then $\wx\cd$~is an admissible trajectory for the convex system
\eqref{12},\,\eqref{13} for $k=n+1$.
\end{corollary}

\begin{proof} \rm
The first assertion is a clear corollary to Theorem~\ref{T3}. The second assertion follows from the proof of
the last assertion of Theorem~\ref{T2}, because it involves only the existence of a~sequence
of admissible trajectories converging to~$\wx\cd$.
\end{proof}

\section*{Examples}

In this section we give examples illustrating the results of the first section.

\subsection*{Example 1}
Consider the following optimal control problem
\begin{multline}\label{exam2}
x_2(1)\to\inf,\quad \dot x_1=u(t),\quad \dot x_2=(x_1-f(t))^2+u^2(t),\quad u(t)\in U,\\
x_1(0)=x_2(0)=0,\quad x_1(1)=f(1),
\end{multline}
where a function $f\colon [0,1]\to\mathbb R$ is absolutely continuous, $f(0)=0$, $|\dot f(t)|\le1$
and $|\dot f(t)|\ne1$ for almost all $t\in[0,1]$ and $U=(-\infty,-1]\cup[1,+\infty)$.

To deal with this problem, we first follow the standard approach, namely, we try to find an optimal
trajectory with the help of the Pontryagin maximum principle. By examining the
conditions of this principle, we show that there is no optimal trajectory in this problem.
Next, using Theorem~\ref{T2}, we find a~function ``suspected'' for a~local infimum. And finally,
by an appeal to Theorem~\ref{T3} we show that the function thus obtained is a~local infimum in problem~\eqref{exam2}.

So, let us assume that $\wx\cd=(\wx_1\cd,\wx_2\cd)$ is an optimal trajectory in problem
\eqref{exam2}; that is, there exists $\wu\cd\in L_\infty([0,1])$ such that the pair
$(\wx\cd,\wu\cd)$ is admissible for this problem and delivers a~strong minimum in it. Then by the
Pontryagin maximum principle, there exist a~nonzero absolutely continuous vector function
$p\cd=(p_1\cd, p_2\cd)$ and a~number $\lambda_0\ge0$ such that
\begin{equation}\label{exm4}
\dot p_1(t)=-2 p_2(t)(\wx_1(t)-f(t)),\quad \dot p_2(t)=0,\quad p_2(1)=-\lambda_0
\end{equation}
and
\begin{equation}\label{exm5}
\max_{u\in U}(p_1(t)u+p_2(t)u^2)=p_1(t)\wu(t)+p_2(t)\wu_2(t)
\end{equation}
for almost all $t\in[0,1]$.

Let us show that this implies the equality $\wu\cd=\dot f\cd$, which is contradictory, because by the condition
$|\wu(t)|\ge1$ for almost all $t\in[0,1]$, and $|\dot
f(t)|<1$ on a~set of positive measure. This will prove that problem \eqref{exam2}
has no optimal trajectory.

From \eqref{exm4} it follows that  the function $p_2\cd$ is constant. This constant is nonzero, because
if $p_2=0$, then from \eqref{exm4} it follows that $p_1\cd$ is a~nonzero constant, and in this case
equality \eqref{exm5} is clearly impossible.  We set $p_2=-1$.

From \eqref{exm5} with $u=1$ we see that $p_1(t) -1\le p_1(t)\wu(t)-\wu^2(t)$ for almost all
$t\in[0,1]$, or what is the same
\begin{equation*}
p_1(t)(\wu(t)-1)\ge\wu^2(t)-1\ge0.
\end{equation*}
Similarly, from \eqref{exm5} for $u=-1$ we find that
\begin{equation*}
p_1(t)(\wu(t)+1)\ge\wu^2(t)-1\ge0.
\end{equation*}
From the first inequality we get
\begin{equation}\label{exm6}
p_1(t)(\wu(t)-\dot f(t))=p_1(t)(\wu(t)-1)+p_1(t)(1-\dot f(t))
\ge p_1(t)(1-\dot f(t))
\end{equation}
and from the second one, we find that
\begin{equation}\label{exm7}
p_1(t)(\wu(t)-\dot f(t))\ge p_1(t)(-1-\dot f(t)).
\end{equation}
In turn, from \eqref{exm6} and \eqref{exm7} we find that, for almost all $t\in[0,1]$,
\begin{equation}\label{exm8}
p_1(t)(\wu(t)-\dot f(t))\ge0.
\end{equation}
Indeed, if  on some set of positive measure the function $p_1\cd$
is nonnegative, then in view of the properties of~$f\cd$ inequality \eqref{exm8} readily follows from
\eqref{exm6}, and if it is nonpositive, then \eqref{exm8} follows from \eqref{exm7}.

From \eqref{exm8},  \eqref{exm4} and from the boundary conditions in problem \eqref{exam2} it follows that
($\wu\cd=\dot \wx_1\cd$)
\begin{multline*}
0\le\int_0^1p_1(t)(\dot\wx_1(t)-\dot f(t))\,dt=p_1(t)(\wx_1(t)-f(t))|_0^1\\- \int_0^1\dot
p_1(t)(\wx_1(t)-f(t))\,dt=-2\int_0^1(\wx_1(t)-f(t))^2\,dt\le0,
\end{multline*}
that is, $\wx_1(t)=f(t)$ for all $t\in [t_0,t_1]$, and therefore, $\wu(t)=\dot\wx_1(t)=\dot f(t)$
for almost all $t\in[0,1]$. But, as was already noted, this is impossible, and hence problem
\eqref{exam2} has no optimal trajectory.

So, the Pontryagin maximum principle gives nothing for the problem under consideration.
Let us employ Theorem~\ref{T2} to find a~function delivering a~local
infimum in problem \eqref{exam2}. Applying this theorem with $k=2$, we conclude that if
$\wx\cd=(\wx_1\cd,\wx_2\cd)$ is a~local infimum, then, for any measurable function $\widehat\alpha\cd$, $0\le\alpha(t)\le1$,
for almost all $t\in[t_0,t_1]$ and any functions $\wu_i\cd\in L_\infty([0,1])$, $\wu_i(t)\in U$
for almost all $t\in[t_0,t_1]$, $i=1,2$, such that
\begin{equation}\label{exam3}
\begin{aligned}
\dot{\wx}_1(t)&=(1-\widehat \alpha(t))\wu_1(t)+\widehat\alpha(t)\wu_2(t),\\
\dot{\wx}_2(t)&=(1-\widehat \alpha(t))\wu^2_1(t)+\widehat\alpha(t)\wu^2_2(t)+
(\wx_1(t)-f(t))^2,
\end{aligned}
\end{equation}
$\wx_1(0)=\wx_2(0)=0$ and $\wx_1(1)=f(1)$, there exist a~nonzero absolutely continuous
vector function $p\cd$ and a~number $\lambda_0\ge0$ such that
\begin{equation}\label{exam4}
\dot p_1(t)=-2 p_2(t)(\wx_1(t)-f(t)),\quad \dot p_2(t)=0,\quad p_2(1)=-\lambda_0
\end{equation}
and moreover, for almost all $t\in[0,1]$,
\begin{equation}\label{exam5}
\max_{u\in
U}(p_1(t)u+p_2(t)((\wx_1(t)-f(t))^{2}+u^2)=p_1(t)\dot\wx_1(t)+p_2(t)\dot\wx_2(t).
\end{equation}

It is seen that relations \eqref{exam4} coincide with \eqref{exm4} and are independent of
$\wu_1\cd$, $\wu_2\cd$ and $\widehat\alpha\cd$.
In this case, our problem becomes simpler: one needs to find functions $\wx_1\cd$ and $\wx_2\cd$
satisfying the convex system \eqref{exam3} (at least for one tuple $\wu_1\cd$, $\wu_2\cd$ and
$\widehat\alpha\cd$) and such that there exist a~nonzero absolutely continuous vector function
$p\cd=(p_1\cd, p_2\cd)$ and a~number $\lambda_0\ge0$ satisfying
\eqref{exam4} and~\eqref{exam5}.

From the above relations, repeating in essence the previous arguments in the proof of
inequalities \eqref{exm6} and \eqref{exm7}, we get the inequalities
\begin{equation}\label{exam6n}
p_1(t)(\dot\wx_1(t)-\dot f(t))\ge p_1(t)(1-\dot f(t))
\end{equation}
and
\begin{equation}\label{exam7n}
p_1(t)(\dot\wx_1(t)-\dot f(t))\ge p_1(t)(-1-\dot f(t))
\end{equation}
for almost all $t\in[0,1]$. Using these inequalities we find, as before, that
\begin{equation*}
p_1(t)(\dot\wx_1(t)-\dot f(t))\ge0
\end{equation*}
for almost all $t\in[0,1]$.

Further, repeating now verbatim the above arguments, we see that  $\wx_1(t)=f(t)$ for all
$t\in [t_0,t_1]$. Now from \eqref{exam4} it follows that the function $p_1\cd$ is constant.
Since $|\dot f(t)|\ne1$ for almost all $t\in[0,1]$, the equalities  $\dot f(t)=1$ and $\dot
f(t)=-1$ are impossible for almost all $t\in[0,1]$, and hence from inequalities \eqref{exam6n} and
\eqref{exam7n} we get $p_1=0$.

We have $\wx_1(t)=f(t)$, $t\in[0,1]$, and hence from \eqref{exam5} for  $u=1$ we get the inequality
$\dot \wx_2(t)\le 1$ for almost all $t\in[0,1]$. On the other hand, by the second of
\eqref{exam3} and the definition of the set~$U$, we conclude that $\dot \wx_2(t)\ge 1$ for almost all
$t\in[0,1]$; that is, $\dot \wx_2(t) = 1$ a.e.\ on $[0,1]$, and so, $\wx_2(t) = t$.

So, for any functions $\wx\cd$, $\wu_1\cd$, $\wu_2\cd$
and $\widehat\alpha\cd$ admissible for convex system, we get a~unique trajectory $\wx(t) =(f(t),t)$,
$t\in[0,1]$, which is suspicious for a~local infimum in problem \eqref{exam2} and which is
not admissible for this problem (for otherwise there should exist a~control
$u\cd$ such that $|u(t)|\ge1$ and $\dot\wx_1(t)=\dot f(t)=u(t)$ for almost all
$t\in[0,1]$, but this is impossible, as was already pointed out).
Moreover, as $\wu_1\cd$,
$\wu_2\cd$ and $\widehat\alpha\cd$ one can take $\wu_1(t)=1$, $\wu_2(t)=-1$ and
$\widehat\alpha(t)=(1-\dot f(t))/2$ for almost all $t\in[0,1]$.

Let us show that the trajectory thus found is a~global infimum in problem  \eqref{exam2}.
To this end, we employ Theorem~\ref{T3}. The regularity of the convex system \eqref{exam3} at
a~point  $(\wx\cd$, $\wu_1\cd$, $\wu_2\cd$, $\widehat\alpha\cd$) means that the relations
\eqref{exam4} and \eqref{exam5} with $\lambda_0=0$ are satisfied only by the zero
vector function $p\cd=(p_1\cd,p_2\cd)$. But this is indeed so: it is clear that  $p_2\cd=0$,
and moreover, that $p_1\cd=0$ was already proved above.

Therefore, by Theorem~\ref{T3} the trajectory $\wx\cd$ lies in the closure of the admissible trajectories for problem  \eqref{exam2}.
Next, for any admissible trajectory $x\cd=(x_{1}\cd, \,x_{2}\cd)$ we have
\begin{equation*}
x_2(1)=\int_{0}^1((x_1(t)-f(t))^2+u^2(t))\,dt\ge1=\wx_2(1)
\end{equation*}
and hence  $\wx\cd$ is a~global infimum for problem \eqref{exam2}.

Note that the trajectory $\wx\cd$ is a sliding regime for the system specifying the constraints in problem~\eqref{exam2}.

One can easily construct a~sequence of admissible trajectories
$x_n\cd=(x_{1n}\cd,\, x_{2n}\cd)$ for problem \eqref{exam2} such that $x_{2n}(1)\to 1$ as $n\to\infty$. Let $n\in\mathbb N$.
We split the interval $[0,1]$ into~$n$ intervals:
$[s/n,\,(s+1)/n]$, $s=0,\ldots,n-1$. We set $b_n(s)=f(s/n)-(s/n)$ and
$c_n(s)=f((s+1)/n)+(s+1)/n$, $s=0,\ldots,n-1$. It can be easily checked that
$((c_n(s)-b_n(s))/2)\in [s/n,(s+1)/n]$, $s=0,\ldots,n-1$.

Consider the sequence $x_{1n}\cd$ defined by
\begin{equation*}
x_{1n}(t)=\begin{cases} t+b_n(s),\qquad t\in[s/n, \,(c_n(s)-b_n(s))/2],\\[10pt]
-t+c_n(s),\quad \ t\in[(c_n(s)-b_n(s))/2, \,(s+1)/n],
\end{cases}
\end{equation*}
$s=0,\ldots,n-1$.

This is a broken line (with slopes $\pm1$ between its segments and which interpolates $f\cd$ at the points
$s/n$, $s=0,\ldots,n$), which converges uniformly to~$f\cd$. We set $u_n\cd=\dot x_{1n}\cd$ and
\begin{equation*}
x_{2n}(t)=\int_{0}^t((x_{1n}(\tau)-f(\tau))^2+u_n^2(\tau))\,d\tau,\quad t\in[t_0,t_1]
\end{equation*}
Since $|u_{n}(t)|=1$ for almost all $t\in[t_0,t_1]$, the pairs  $(x_{1n}\cd,\,x_{2n}\cd)$,
$n\in\mathbb N$, are admissible for problem \eqref{exam2}. Moreover, it is clear that $x_{2n}(1)\to 1$ as $n\to\infty$.

\subsection*{Example 2}
Here we give an example when from Theorem~\ref{T2} one can derive more information about
an optimal process in comparison with that delivered by the Pontryagin maximum principle.

Let $g\colon \mathbb R\to\mathbb R$ and $U\subset \mathbb R$. Consider the problem
\begin{equation}\label{z1}
\int_{0}^{1}x g(u)\,dt\to\inf,\quad \dot x=u(t),\quad x(0)=x(1)=0,\quad u(t)\in U.
\end{equation}
Assume that the function $g$ is continuous, $g(0)=0$, and $0\in{\rm int}\,U$.

In an equivalent form, this problem reads as
\begin{multline}\label{z2}
x_2(1)-x_2(0)\to\inf,\quad \dot x_1=u(t),\quad\dot x_2=x_1g(u(t)),\\
x_1(0)=x_1(1)=0,\quad u(t)\in U.
\end{multline}
Let us show that the equilibrium point $\wx_1\cd=\wx_2\cd=0$, $\wu\cd=0$ satisfies
the Pontryagin maximum principle for problem \eqref{z2}; hence this point is ``suspicious'' from the
viewpoint of this principle for a~strong minimum in this problem. Indeed, the adjoint equation and
the maximum condition at this point are equivalent to the relations
$$
\dot p_1(t)=-p_2(t)g(0)=0,\qquad \dot p_2(t)=0, \, \quad p_2(0)=-p_2(1)=\lambda_0
$$
and
$$
p_1(t)u\le0,\,\,\,\forall\,u\in U.
$$
It is clear that $p_1\cd$ is a~constant, and  moreover, since $0\in{\rm int}\,U$, this constant is zero.
The function $p_2\cd$ is also constant. Setting, for example, $p_2=-\lambda_0=-1$, we conclude that the
point $\wx_1\cd=\wx_2\cd=0$, $\wu\cd=0$ satisfies the Pontryagin maximum principle.

We now employ Theorem~\ref{T2} to show that if a~point $\wx_1\cd=\wx_2\cd=0$, $\wu\cd=0$
is a~point of strong minimum for problem \eqref{z1}, then certain additional meaningful conditions should be satisfied.
Namely, the following result holds.

\begin{proposition}
If a point $\wx_1\cd=\wx_2\cd=0$, $\wu\cd=0$  delivers a~strong minimum for problem
\eqref{z1}, then
the function $u\mapsto g(u)$ is linear on some interval with center at the origin.
\end{proposition}

\begin{proof} \rm
We apply Theorem \ref{T2} at the point $\wx_1\cd=\wx_2\cd=0$. This means, in particular, that
for any $u_i\in U$, $i=1,2$, and $\alpha \in [0,1]$ such that
\begin{equation}\label{z3}
0=(1-\alpha)u_1+\alpha u_2, \quad 0=(1-\alpha) 0g(u_1)+\alpha 0g(u_2)
\end{equation}
there exsit a~vector function $p\cd=(p_1\cd, p_2\cd)$ and a~number  $\lambda_0\ge0$
satisfying the relations
\begin{equation}\label{z4}
\dot p_1(t)=-p_2(t)((1-\alpha)g(u_1)+\alpha g(u_2)),\quad \dot p_2(t)=0, \quad
p_2(0)=p_2(1)=-\lambda_0
\end{equation}
and
$$
p_1(t)u\le0,\,\,\,\forall\,u\in U.
$$
We have $0\in{\rm int}\,U$, and hence this implies, as before, that $p_1\cd$ is the
constantly zero. As a~result, $p_2\ne0$, for otherwise all the Lagrange multipliers would be zero.

Let $\varepsilon>0$ be such that $[-\varepsilon,\varepsilon]\subset U$. In this case it is clear
that any $u_1\in[-\varepsilon,0)$,  $u_2\in(0,\varepsilon]$   and
$\alpha=u_1/(u_1-u_2)\in(0,1)$ would satisfy the equations from \eqref{z3}. Now from the first equality in~\eqref{z4} we get
the relation
\begin{equation}\label{z5}
u_2g(u_1)=u_1g(u_2),
\end{equation}
which holds for any $u_1\in[-\varepsilon,0)$ and $u_2\in(0,\varepsilon]$.

Making   $u_1=-\varepsilon$, $u_2=\varepsilon$ in \eqref{z5},  we see that
\begin{equation}\label{z6}
g(-\varepsilon)=-g(\varepsilon).
\end{equation}

Let $u\in[-\varepsilon,\varepsilon]$. If $u<0$, then from \eqref{z5} with $u_1=u$ and
$u_2=\varepsilon$ we get
$$
g(u)=\frac{g(\varepsilon)}{\varepsilon}\,u.
$$

If $u>0$, then by another appeal to \eqref{z5} with $u_2=u$ and $u_1=-\varepsilon$ and taking into account \eqref{z6}
we obtain
$$
g(u)=-\frac{g(-\varepsilon)}{\varepsilon}\,u=\frac{g(\varepsilon)}{\varepsilon}\,u.
$$

If $u=0$, then by the hypothesis $g(0)=0$, and hence
$$
g(u)=\frac{g(\varepsilon)}{\varepsilon}\,u, \quad\forall\,u\in[-\varepsilon,\varepsilon].
$$
\end{proof}

So, Theorem~\ref{T2} strengthens in general the Pontryagin maximum principle.

\subsection*{Example 3} This examples shows that the condition that the set~$U$
be compact in the last assertion of Theorem~\ref{T2} and in the second part of Corollary \ref{sl2}
is essential.

Consider the control system
\begin{equation}\label{zz}
\dot x=u(t),\quad u(t)\in\mathbb R,\quad x(0)=0,\quad x(1)=1
\end{equation}
and construct the following sequence of functions
\begin{equation}\label{zz1}
x_n(t)=
\begin{cases}
\sqrt{n}\,t,& t\in[0,1/n],\\[5pt]
\sqrt{t},& t\in[1/n,1].
\end{cases}
\end{equation}
It is clear that this is a sequence of absolutely continuous functions which are admissible for system
\eqref{zz} and which converge uniformly on $[0,1]$ to the function $t\mapsto\sqrt{t}$,
whose derivative, clearly, does not lies in $L_\infty([0,1])$.

It follows that   the equality
\begin{equation*}
\dot x(t)=(1-\alpha(t))u_1(t)+\alpha(t) u_2(t),
\end{equation*}
where $x(t)=\sqrt{t}$, $t\in[0,1]$, cannot be satisfied for almost all $t\in[0,1]$
for any measurable function $\alpha\cd$ for which
$0\le\alpha(t)\le1$ for almost all $t\in[0,1]$ and for any functions $u_i\cd\in L_\infty([0,1])$, $i=1,2$.
This means that the
trajectory $x\cd$, which lies in the closure of the admissible trajectories for system \eqref{zz},
is not admissible for any convex extension of this system for $k=2$.

\subsection*{Example 4} Here our aim is to show that the regularity condition
in Corollary~\ref{co1} and in the first part of Corollary~\ref{sl2} is essential.

Consider the problem
\begin{multline}\label{zz2}
f_0(x_1(1))=x_1^2(1)\to\inf,\quad \dot x_1=u(t),\quad \dot x_2=4u^2(t)-3u^3(t),\\ \dot
x_3=(x_1-x_2)^2,\quad x_1(0)=x_2(0)=x_3(0)=x_3(1)=0,\\ u(t)\in U=\{-1,\, 1/3,\, 1,\, 3\}.
\end{multline}
If a trajectory $x\cd=(x_1\cd, x_2\cd, x_3\cd)$ is admissible for the system specifying
the constraints in this problem, then from the third differential equation and the boundary
conditions it follows that $x_1\cd=x_2\cd$. Now from the first and second differential equations
we conclude that, for all $t\in[0,1]$, the equality holds
$\int_0^tu(\tau)\,d\tau=\int_0^t(4u^2(\tau)-3u^3(\tau))\,d\tau$ for some  $u\cd\in
L_\infty([0,1])$ for which $u(t)\in U$ for almost all $t\in[0,1]$. It follows that
$u(t)=4u^2(t)-3u^3(t)$ for almost all $t\in[0,1]$. The control $u\cd$ cannot assume
values not lying in~$U$ on a~set of positive measure. Hence from the last equality it follows that $u(t)$
is either  $1/3$ or~$1$ for almost all $t\in[0,1]$. This implies, in particular, that
$x_1(1)=\int_0^1u(t)\,dt\ge1/3$ for any admissible trajectory.

Let us show that the zero trajectory $\wx\cd=(0,0,0)$  delivers a~global minimum for the convex extension of
problem \eqref{zz2} with $k=3$. Indeed, a~direct verification shows that the triple $(\wx\cd,\,\widehat{\ov u}\cd,\,\widehat{\ov
\alpha}\cd)$, where $u_1\cd=-1$, $u_2\cd=3$, $u_3\cd=1/3$,  $\alpha_1\cd=3/8$,
$\alpha_2\cd=1/16$, and $\alpha_3\cd=9/16$, is admissible for this extension and that the zero
delivers a~global minimum $f_0$. But this trajectory is neither a~local infimum nor a~sliding regime for problem \eqref{zz2},
because by the above estimate $x_1(1)\ge1/3$ this trajectory does not lie in the closure
of the set of admissible trajectories for this problem.

So, the assertions of Corollary~\ref{co1} and the first part of Corollary~\ref{sl2}
are not true for the case under consideration. This can be explained by the fact that the regularity
condition is violated\,---\,namely, the convex extension of the system specifying the constraints in problem
\eqref{zz2} is not regular at the point $\wx\cd$. Indeed, by definition, the regularity
is equivalent to saying that only the zero vector function
$p\cd=(p_1\cd, p_2\cd, p_3\cd)$ can satisfy the relations
\begin{equation*}
\dot p_1\cd=\dot p_2\cd=\dot p_3\cd=0,\quad p_1(1)=p_2(1)=0
\end{equation*}
and
\begin{equation*}
p_3(t)0\le 0.
\end{equation*}
But this implies that $p_1\cd$ and $p_2\cd$ are zero constants and as $p_3\cd$
one can take any nonzero constant.

Note that in this case there exists an optimal trajectory\,---\,namely, putting
$u\cd=1/3$ we see that  $\wx_1(1)=1/3$.

\section*{Appendix}

In this section we prove the generalized implicit function theorem and establish four lemmas:
the inverse function lemma, the lemma on equation in variations, and two approximation lemmas.

We introduce the following definition. Let $X$ and $Y$ be normed spaces, $\Sigma$
be a~topological space, and let~$M$ be a~nonempty subset of~$X$. We denote by
$C^1_x(M\times\Sigma, Y)$ the restriction to $M\times\Sigma$ of the set
of mappings $F\colon X\times\Sigma\to Y$ which are continuous together with its derivative with respect to~$x$
and for which the norm
$$
\|F\|_{C^1_x(M\times\Sigma, Y)}=\sup\limits_{(x,\sigma)\in
M\times\Sigma}\|F(x,\sigma)\|_Y+\sup\limits_{(x,\sigma)\in
M\times\Sigma}\|F_x(x,\sigma)\|
$$
is finite.

\begin{theorem}[\rm the generalized implicit function theorem]
Let $X$ and $Y$ be Banach spaces, $\Sigma$~be a~topological space,
$\ws\in\Sigma$, $V$~be a~neighborhood of a~point $\wx\in X$, $Q$~be a~convex closed
subset of~$X$, $\wF\in C^1_x((V\cap Q)\times\Sigma,\, Y)$,  $\wF(\wx,\ws)=0$, and let
the operator $\wF_x(\wx,\ws)$ be invertible.

Then there exist neighborhoods  $V_0'\subset V_0\subset V$ of the point $\wx$, a~neighborhood $U_0$
of~$\ws$, and a~neighborhood $W_0$ of the mapping~$\wF$ such that, for $F\in W_0$ for which
$x-\wF^{-1}_x(\wx,\ws)F(x,\sigma)\in Q$ for all $(x,\sigma)\in (V'_0\cap
Q)\times U_0$, there exists a~continuous mapping $g_F\colon U_0\to V_0\cap Q$ such that
\begin{equation}\label{i1}
F(g_F(\sigma),\sigma)=0\,\,\, \text{and}\,\,\, \|g_F(\sigma)-x\|_X\le
2\|(\wF_x(\wx,\ws))^{-1}\|\|F(x,\sigma)\|_Y
\end{equation}
for all $(x,\sigma)\in (V'_0\cap Q)\times U_0$. Moreover, the equality $F(x,\sigma)=0$ on
$(V_0\cap Q)\times U_0$ is possible only if $x=g_F(\sigma)$.
\end{theorem}

\begin{proof} \rm
For brevity, we set $\Lambda=\widehat F_{x}(\wx,\ws)$ and write $C^1_x$ in place of
$C^1_x((V\cap Q)\times\Sigma, Y)$. The mapping $(x,\sigma)\mapsto \widehat F_x(x,\sigma)$
is continuous at the point $(\wx,\ws)$, and hence there exist $0<\delta\le1$ such that
$U_X(\wx,\delta)\footnote{$U_X(\wx,\delta)$ denotes the open ball in a~normed space~$X$
with center at~$\wx$ and of radis $\delta$.}\subset V$ and a~neighborhood~$U$
of the point~$\ws$ for which $\|\wF_x(x,\sigma)-\Lambda\|\le1/8\|\Lambda^{-1}\|$ for all
$(x,\sigma)\in U_X(\wx,\delta)\times U$.

We set $V_0=U_X(\wx,\delta)$, and choose neighborhoods $V'_0$, $U_0$ and $W_0$ so that
$V'_0\subset U_X(\wx,\delta/2)$, $U_0\subset U$, and moreover,
$\|\wF(x,\sigma)\|_Y<\delta/8\|\Lambda^{-1}\|$ if $(x,\sigma)\in V'_0\times U_0$,
$W_0=U_{C^1_x}(\wF,\delta/8\|\Lambda^{-1}\|)$.

Let $F\in W_0$ and $x-\wF^{-1}_x(\wx,\ws)F(x,\sigma)\in Q$ for all $(x,\sigma)\in
(V'_0\cap Q)\times U_0$. We claim that, for any $x,x'\in V_0\cap Q$ and $\sigma\in U_0$,
\begin{equation}\label{j0}
\|F(x,\sigma)-F(x',\sigma)-\Lambda(x-x')\|_{Y} \le \frac1{2\|\Lambda^{-1}\|}\|x-x'\|_{X}.
\end{equation}

Indeed, first, we have  ($\delta\le1$)
\begin{multline}\label{j1}
\|F_x(x,\sigma)-F_x(\wx,\ws)\|\le\|F_x(x,\sigma)-\wF_x(x,\sigma)\|+\|\wF_x(x,\sigma)-\wF_x(\wx,\ws)\|\\
+\|F_x(\wx,\ws)-\wF_x(\wx,\ws)\|<\frac3{8\|\Lambda^{-1}\|}\,.
\end{multline}
The sets $V_0$ and $Q$ are convex, and hence if $x,x'\in V_0\cap Q$, then
$x_\theta=(1-\theta)x+\theta x'\in V_0\cap Q$ for $\theta\in[0,1]$. By the mean value theorem,
as applied to the mapping $x\to F(x,\sigma)-F_x(\wx,\ws)x$, where $\sigma\in U_0$,
we get, by \eqref{j1} and in view of the choice of~$F$, that
\begin{multline*}
\|F(x,\sigma)-F(x',\sigma)- \Lambda(x-x')\|_{Y}\le\|F(x,\sigma)-F(x',\sigma)\\-
F_x(\wx,\ws)(x-x')\|_{Y}+\|F_x(\wx,\ws)(x-x')-\Lambda(x-x')\|_Y\\\le
\sup_{\theta\in[0,1]}\|F_x(x_\theta,\sigma)-F_x(\wx,\ws)\|
\|x-x'\|_{X}+\|F_x(\wx,\ws)-\Lambda\|\|x-x'\|_X\\\le\frac3{8\|\Lambda^{-1}\|}\|x-x'\|_{X}+
\frac1{8\|\Lambda^{-1}\|}\|x-x'\|_{X}=\frac1{2\|\Lambda^{-1}\|}\|x-x'\|_{X},
\end{multline*}
thereby proving  inequality \eqref{j0}.

Let $(x,\sigma)\in (V'_0\cap Q)\times U_0$. Considering the sequence (the modified Newton method)
\begin{equation}\label{i3}
x_n=x_{n-1}-\Lambda^{-1}F(x_{n-1},\sigma),\quad n\in\mathbb N,\quad x_0=x,
\end{equation}
we claim that this sequence lies in $U_X(\wx,\delta)\cap Q$ and is a~Cauchy sequence.
The first claim is proved by induction. It is clear that $x_0\in U_X(\wx,\delta)\cap Q$. Let $x_k\in
U_X(\wx,\delta)\cap Q$, $1\le k\le n$. We need to show that $x_{n+1}\in U_X(\wx,\delta)\cap Q$.

Applying the operator $\Lambda$ to both sides of \eqref{i3}, we find that
\begin{equation}\label{i4}
\Lambda(x_n-x_{n-1})=-F(x_{n-1},\sigma).
\end{equation}
Using in succession \eqref{i3}, \eqref{i4}, \eqref{j0} and then iterating, we find that
\begin{multline}\label{i5}
\|x_{n+1}-x_n\|_X\le \|\Lambda^{-1}\|\|F(x_n,\sigma)\|_Y=\|\Lambda^{-1}\|
\|F(x_n,\sigma)-F(x_{n-1},\sigma)-\\
-\Lambda(x_n-x_{n-1})\|_Y\le\frac 12\|x_n-x_{n-1}\|_X\le {\ldots}\le\frac
1{2^n}\|x_1-x\|_X.
\end{multline}
Next, employing the triangle inequality, using \eqref{i5}, \eqref{i3}, and
taking into account the formula for the sum of a geometric progression, we have, since $F\in W_0$,
\begin{multline}\label{i6}
\|x_{n+1}-\wx\|_X\le\|x_{n+1}-x\|_X+\|x-\wx\|_X \le
\|x_{n+1}-x_n\|_X\\+\ldots+\|x_1-x\|_X+\|x-\wx\|_X\le \left(\frac 1{2^n}+\frac
1{2^{n-1}}+\ldots+1\right)\|x_1-x\|_X\\+\|x-\wx\|_X<
2\|\Lambda^{-1}\|\|F(x,\sigma)\|_Y+\|x-\wx\|_X\\<2\|\Lambda^{-1}\|\|F(x,\sigma)-\wF(x,\sigma)\|_Y
+2\|\Lambda^{-1}\|\|\wF(x,\sigma)\|_Y+\|x-\wx\|_X\\<\frac{\delta}{4}+\frac{\delta}{4}+\frac{\delta}{2}=\delta,
\end{multline}
that is, $x_{n+1}\in U_X(\wx,\delta)$.

By the induction hypothesis,  $x_n\in U_X(\wx,\delta)\cap Q$, and so
$x_{n+1}=x_n-\Lambda^{-1}F(x_n,\sigma)\in Q$ by the choice of the mapping~$F$.
Therefore, the entire sequence $\{x_n\}$ lies in $U_X(\wx,\delta)\cap Q$.

Next, the sequence $\{x_n\}$ is a Cauchy sequence. Indeed, using  \eqref{i5} and
arguing as in the previous inequality, we have, for all $n,m\in\mathbb N$,
\begin{multline}\label{i7}
\|x_{n+m}-x_n\|_X\le\|x_{n+m}-x_{n+m-1}\|_X+\ldots+
\|x_{n+1}-x_n\|_X\le\\
\le\left(\frac 1{2^{n+m-1}}+\ldots+\frac 1{2^n}\right)\|x_1-x\|_X <
\frac{\|x_1-x\|_X}{2^{n-1}} < \frac{\delta}{2^n}\,,
\end{multline}
which proves that  $\{x_n\}$ is a Cauchy sequence.

The functions $x_n$ are defined on $(V'_0\cap Q)\times U_0$. Let $(x,\sigma)\in (V'_0\cap
Q)\times U_0$. We set $\widetilde g_F(x,\sigma)=\lim_{n\to\infty}x_n$. From \eqref{i6}
it follows that  $\widetilde g_F(x,\sigma)\in U_X(\wx,\delta)=V_0$. Since the set~$Q$
is closed, we have $\widetilde g_F(x,\sigma)\in Q$, and thus the mapping
$\widetilde g_F\colon (V'_0\cap Q)\times U_0\to (V_0\cap Q)$ is defined.

Making $n\to\infty$ in \eqref{i4}  and taking into account that $F$ is continuous, we get the relation
 $F(\widetilde g_F(x,\sigma),\sigma)=0$.

Let us show that $\widetilde g_F(x,\sigma)=\widetilde g_F(\wx,\sigma)$ for any point
$(x,\sigma)\in (V'_0\cap Q)\times U_0$. Indeed, by~\eqref{j0} we have
\begin{multline}\label{rav}
\|\widetilde g_F(x,\sigma)-\widetilde g_F(\wx,\sigma)\|_X=
\|\Lambda^{-1}\Lambda(\widetilde g_F(x,\sigma)-\widetilde g_F(\wx,\sigma))\|_X
\\\le\|\Lambda^{-1}\|\|\Lambda(\widetilde g_F(x,\sigma)-\widetilde g_F(\wx,\sigma))\|_Y=
\|\Lambda^{-1}\|\|F(\widetilde g_F(x,\sigma),\sigma)\\-F(\widetilde
g_F(\wx,\sigma),\sigma)- \Lambda(\widetilde g_F(x,\sigma)-\widetilde g_F(\wx,\sigma)\|_Y
\\\le\frac12\|\widetilde g_F(x,\sigma)-\widetilde g_F(\wx,\sigma)\|_X,
\end{multline}
that is, $\widetilde g_F(x,\sigma)=\widetilde g_F(\wx,\sigma)$.

We set $g_F(\sigma)=\widetilde g_F(\wx,\sigma)$. This is a~mapping from  $U_0$ into $V_0\cap Q$
and by the above $F(g_F(\sigma),\sigma)=0$ for all $\sigma\in U_0$.

From \eqref{i6} it follows that  $\|x_{n}-x\|_X\le2\|\Lambda^{-1}\|\|F(x,\sigma)\|_Y$. Making
$n\to\infty$, we get the inequality $\|g_F(\sigma)-x\|_X\le
2\|\Lambda^{-1}\|\|F(x,\sigma)\|_Y$.

Since $F$ is continuous we derive from \eqref{i3} that the functions $x_n$, \textit{qua}
functions  of~$\sigma$, are continuous on~$U_0$. Making $m\to\infty$ in \eqref{i7}, we conclude that
the mapping $\sigma\mapsto g_F(\sigma)$ is the uniform limit of continuous functions and hence is continuous.

That the equality  $F(x,\sigma)=0$ on $(V_0\cap Q)\times U_0$ is possible only when
$x=g_F(\sigma)$ is proved by the same arguments as in~\eqref{rav}.
\end{proof}

To derive another corollary to this theorem we first need one definition.

Let $\mathcal M$ be a topological space,  $Z$ be a normed space.
We denote by $C(\mathcal M,\, Z)$ the space of continuous bounded mappings
$F\colon \mathcal M\to Z$ equipped with the norm $$\|F\|_{C(\mathcal M,\, Z)}=\sup\limits_{x\in
\mathcal M}\|F(x)\|_Z.$$

\begin{corollary}\label{s3}
Let in the theorem the mapping  $\wF$, together with the mapping $F\in W_0$, be such that
$x-\wF^{-1}_x(\wx,\ws)\wF(x,\sigma)\in Q$ for all $(x,\sigma)\in (V'_0\cap Q)\times
U_0$.

Then there exist continuous mappings $g_{F}\colon U_0\to V_0\cap Q$ and $g_{\wF}\colon
U_0\to V_0\cap Q$ such that  $F(g_{F}(\sigma),\sigma)=0$ and
$\wF(g_{\wF}(\sigma),\sigma)=0$ for all $\sigma\in  U_0$ and there exists a~neighborhood
$U'_0\subset U_0$ of the point~$\ws$ such that
\begin{equation}\label{neq}
\|g_{F}-g_{\wF}\|_{C(U'_0,\,X)}\le2\|(\wF_x(\wx,\ws))^{-1}\|\|F-\wF\|_{C((V\cap
Q)\times\Sigma,\, Y)}.
\end{equation}
Moreover, the equalities $F(x,\sigma)=0$ and $\wF(x,\sigma)=0$ on $(V_0\cap Q)\times U_0$
are possible,  respectively, only when $x=g_F(\sigma)$ and $x=g_{\wF}(\sigma)$.
\end{corollary}

\begin{proof} \rm
All the assertions  of the corollary, except for inequality \eqref{neq}, are direct
consequences of the theorem. Let us prove inequality \eqref{neq}.

Since the mapping $g_{\wF}$ is continuous at $\ws$ (from \eqref{i1} for $\wF$ it follows that
$g_{\wF}(\ws)=\wx$), there exists a~neighborhood $U'_0\subset U_0$ of this point such that
$g_{\wF}(\sigma)\in V'_0\cap Q$ for all $\sigma\in U'_0$. For such~$\sigma$,
substituting $g_{\wF}(\sigma)$
in estimate \eqref{i1} for~$F$ in place of~$x$  and subtracting the zero element
$\wF(g_{\wF}(\sigma),\sigma)$ on the right under the norm sign, we get the inequality
\begin{multline*}
\|g_F(\sigma)-g_{\wF}(\sigma)\|_X\le
2\|(\wF_x(\wx,\ws))^{-1}\|\|F(g_{\wF}(\sigma),\sigma)-\wF(g_{\wF}(\sigma),\sigma)\|_Y\\
\le2\|(\wF_x(\wx,\ws))^{-1}\|\|F-\wF\|_{C((V\cap Q)\times\Sigma,\, Y)}.
\end{multline*}
Taking the supremum over $\sigma\in U'_0$ on the left, we get~\eqref{neq}.
\end{proof}

\begin{lemma}[\rm the inverse function lemma]\label{L1}
Let $X$ be a~Banach space, $K$ be a convex closed subset of~$X$, $V$~be
a~neighborhood of a~point $\ww\in K$ and let $\widehat\Phi\colon V\to\mathbb R^m$. Assume that the following
conditions are satisfied:
\begin{itemize}
\item[$1)$] $\widehat \Phi\in C(V\cap K,\,\mathbb R^m)$,
\item[$2)$] $\widehat \Phi$ is continuously differentiable at the point $\ww$,
\item[$3)$] $0\in{\rm int}\,\widehat \Phi'(\ww)(K-\ww)$.
\end{itemize}
Then there exist a~neighborhood $V_0\subset V$ of the point $\ww$ and constants $r_0>0$ and $\gamma>0$
such that, for any $r\in(0,r_0]$ and any mapping $\Phi\in U_{C(V\cap K,\,\mathbb
R^m)}(\widehat \Phi,r)$, there exists a~mapping $g_\Phi(w,y)\colon (V_0\cap
K)\times\mathbb R^m\to V\cap K$ satisfying
\begin{equation}\label{5pr}
\Phi(g_\Phi(w,y))=y,\qquad \|g_\Phi(w,y)-w\|_X\le\gamma r
\end{equation}
for all $(w,y)\in (V_0\cap K)\times\mathbb R^m$, for which
\begin{equation}\label{5n}
|\widehat\Phi(w)-y|+\|w-\ww\|_X\le r.
\end{equation}
\end{lemma}

Before proceeding with the proof of the lemma, we prove one result, which is a~direct corollary
to the implicit function theorem for inclusions (see the paper~\cite{AM1} by the authors of the present paper).

\begin{proposition}\label{Pred2}
Let $X$ and $Y$ be Banach spaces, $\Lambda\colon X\to Y$ be a~linear
continuous operator, $C$~be a~convex closed subset of~$X$, $x_0\in C$, $y_0=\Lambda x_0$ and
\begin{equation}\label{1pr}
y_0\in {\rm int}\, \Lambda C.
\end{equation}
Then there exist neighborhoods $V_1$, $U_1$, $U_2$, respectively, of the points $x_0$, $y_0$, $0_X$,
a~constant $a>0$, and a~continuous mapping $R\colon V_1\times U_1\times U_2\to X$ such that,
for any $\sigma=(\sigma_1,\sigma_2)\in U_1\times U_2$ and $\xi\in V_1$,
\begin{equation}\label{2pr}
\Lambda R(\xi,\sigma)=\sigma_1,\qquad R(\xi,\sigma)\in\sigma_2+C
\end{equation}
and
\begin{equation}\label{3pr}
\|R(\xi,\sigma)-\xi\|_X\le a(\|\Lambda\xi-\sigma_1\|_Y+{\rm dist}\,(\xi-\sigma_2,C)),
\end{equation}
where ${\rm dist}\,(\xi-\sigma_2,C))$ is the distance from $\xi-\sigma_2$ to the set~$C$.
\end{proposition}

\begin{proof} \rm
Let a mapping $F\colon X\to Y\times X$ act by the rule $F(x,\sigma)=(\Lambda
x-\sigma_1, x-\sigma_2)$. We set $A=(0_Y, C)$ and $\sigma_0=(y_0,0_X)$. Then
$F(x_0,\sigma_0)=(\Lambda x_0-y_0, x_0)=(0_Y,x_0)\in A$.

Let us show that inclusion \eqref{1} implies that
\begin{equation}\label{4pr}
0\in{\rm int}({\rm Im}\,F_x(x_0,\sigma_0)+F(x_0,\sigma_0)-A) ={\rm int}({\rm
Im}(\Lambda,{\rm Id})+(0_Y, x_0-C)),
\end{equation}
where ${\rm Id}$ is the identity operator.

By \eqref{1pr}, there exists $\delta_0>0$ such that $U_Y(y_0,\delta_0)\subset \Lambda
C$. Setting $\delta=\delta_0/(\|\Lambda\|+1)$, we show that $U_{Y\times
X}(0,\delta)\subset{\rm int}({\rm Im}(\Lambda,{\rm Id})+(0_Y,x_0-C))$.

Let $\sigma=(\sigma_1,\sigma_2)\in U_{Y\times X}(0,\delta)$. From the choice of~$\delta$
it follows that  $\sigma_1-\Lambda\sigma_2\in U_Y(0,\delta_0)$, and hence
$y_0+\sigma_1-\Lambda\sigma_2\in U_Y(y_0, \delta_0)\subset \Lambda C$. Therefore, there exists
an element $x_1(\sigma)\in C$ such that $y_0+\sigma_1-\Lambda\sigma_2=\Lambda x_1(\sigma)$.

We set $x(\sigma)=\sigma_2+x_1(\sigma)-x_0$. Hence, using the previous equality, we find that
$\sigma_1=\Lambda\sigma_2+\Lambda x_1(\sigma)-\Lambda
x_0=\Lambda(\sigma_2+x_1(\sigma)-x_0)=\Lambda x(\sigma)$ and
$\sigma_2=x(\sigma)+x_0-x_1(\sigma)\in x(\sigma)+x_0-C$;
that is, $\sigma=(\sigma_1,\sigma_2)\in{\rm Im}(\Lambda,{\rm Id})+(0_Y,x_0-C)$. Therefore,
$U_{Y\times X}(0,\delta)\subset{\rm Im}(\Lambda,{\rm Id})+(0_Y,x_0-C)$, proving inclusion~\eqref{4pr}.

Now all the hypotheses of the implicit function theorem from \cite{AM1} are clearly satisfied,
where $\Sigma=Y\times X$ and $V=X$ (with  $\wx$, $\wy$ and $\ws$ in place of
$x_0$, $y_0$ and~$\sigma_0$, respectively). Relations \eqref{2pr} and \eqref{3pr}
are immediate consequences of this theorem.
\end{proof}

\begin{proof} \rm [\rm of the inverse function lemma]
We set   $C=K-\ww$, $x_0=0$, and put $\Lambda=\widehat \Phi'(\ww)$. Then $y_0=\Lambda x_0
=0\in{\rm int}\,\Lambda C$ by condition~$3)$ of the lemma, and hence the hypotheses of
Proposition~\ref{Pred2} with these data are satisfied. Let $V_1$, $U_1$, $U_2$ be the corresponding
neighborhoods (of the origins in~$X$ and $\mathbb R^m$), and let the constant $a>0$ and a~continuous mapping
$R\colon V_1\times U_1\times U_2\to X$ be from this proposition.

Let $\delta_1>0$ be such that $U_X(0,\delta_1)\subset V_1\cap U_2$. Then if  $w\in
U_X(\ww,\delta_1)$ and $z\in U_1$, then $(\ww-w, (z,\ww-w))\in V_1\times U_1\times U_2$.
Therefore, the mapping $\varphi\colon U_X(\ww,\delta_1)\times U_1\to X$ is defined by the formula
$\varphi(w,z)=R(\ww-w, (z,\ww-w))$. Moreover, by \eqref{2} and \eqref{3},
we have
\begin{equation}\label{7pr}
\Lambda \varphi(w,z)=z,\qquad \varphi(w,z)\in\ww-w+K-\ww=K-w
\end{equation}
and
\begin{equation}\label{8pr}
\|\varphi(w,z)-(\ww-w)\|_X\le a|\Lambda(\ww-w)-z|,
\end{equation}
because ${\rm dist}\,(0,K-\ww)=0$.

From condition $2)$ of the lemma it follows that  there exists $\delta_2>0$ such that
$U_X(\ww,\delta_2)\subset V$ and moreover, for all $w,w'\in U_X(\ww,\delta_2)$,
\begin{equation}\label{9pr}
|\widehat \Phi(w)-\widehat\Phi(w')-\Lambda(w-w')|\le \frac
1{a(\|\Lambda\|+3)+1}\|w-w'\|_X.
\end{equation}
We set $V_0=U_{\mathbb R^m}(\ww,\delta)$, where $\delta=\min(\delta_1,\delta_2)$, and choose
$\rho>0$ so that $B_{\mathbb R^m}(0,\rho)\subset U_1$. Let
$r_0=\min(\rho/3,\,\delta/a(\|\Lambda\|+3))$.

Let $r\in(0,r_0]$, $\Phi\in U_{C(V\cap K,\,\mathbb R^m)}(\widehat \Phi,r)$, and let a~pair
$(w,y)\in (V_0\cap K)\times\mathbb R^m$ satisfy inequality \eqref{5n}. Consider the
mapping $G\colon B_{\mathbb R^m}(\widehat\Phi(x), 3r)\to\mathbb R^m$ defined by
\begin{equation*}
G(z)=y+z-\Phi(w+\varphi(w,z-\widehat \Phi(w))).
\end{equation*}
This definition is correct, because if $z\in B_{\mathbb R^m}(\widehat\Phi(w), 3r)$, then
$|z-\widehat\Phi(w)|\le3r\le3r_0\le\rho$; that is, $z-\widehat\Phi(w)\in U_1$. Next, from
\eqref{8pr}, \eqref{5n} and by the choice of $r_0$ we have
\begin{multline}\label{10pr}
\|\varphi(w,z-\widehat\Phi(w))-(\ww-w)\|_X\le a|\Lambda(\ww-w)-(z-\widehat\Phi(w))|\\\le
\alpha(\|\Lambda\|\|w-\ww\|_X+|z-\widehat\Phi(w)|)\le a(\|\Lambda\|+3)r.
\end{multline}
The right-hand side is smaller than~$\delta$, and hence
$w+\varphi(w, z-\widehat \Phi(w))\in  V$. Besides, from
\eqref{7pr} we get  $w+\varphi(w, z-\widehat \Phi(w))\in  K$.

Let us show that the range of~$G$ lies in the ball $B_{\mathbb R^m}(\widehat\Phi(w), 3r)$.
Indeed,  using the equality $\Lambda \varphi(w,
z-\widehat\Phi(w))=z-\widehat\Phi(w)$, which holds by \eqref{7pr}, and employing
condition \eqref{5n}, relation \eqref{9pr}, the condition $\Phi\in U_{C(V\cap K,\,\mathbb R^m)}(\widehat \Phi,r)$,
inequality \eqref{10pr}, and again condition \eqref{5n}, we get
\begin{multline*}
|G(z)-\widehat\Phi(w)|\le|y-\widehat\Phi(w)|+|z-\Phi(w+\varphi(w,z-\widehat \Phi(w)))|\le r\\
+|\widehat\Phi(w+
\varphi(w,z-\widehat\Phi(w)))-\widehat\Phi(w)-\Lambda\varphi(w,z-\widehat\Phi(w))|\\+
|\Phi(w+ \varphi(w,z-\widehat\Phi(w)))-\widehat\Phi(w+ \varphi(w,z-\widehat\Phi(w)))|\le
r+\\\varepsilon\|\varphi(w,z-\widehat \Phi(w))\|_X+r\le
2r+\varepsilon(\|\varphi(w,z-\widehat \Phi(w))-(\ww-w)\|_X+\|\ww-w\|_X)\\\le
2r+\varepsilon(a(\|\Lambda\|+3)+1)r=3r
\end{multline*}
(here we set, for brevity $\varepsilon=1/(a(\|\Lambda\|+3)+1)$).
The mapping is continuous \textit{qua} the superposition of continuous mappings, and hence by Brouwer's
fixed point theorem, there exists a~point $z_*=z_*(w,y,\Phi)$ such that
$G(z_*)=z_*$ or, what is the same, $\Phi(w+\varphi(w,z_*-\widehat \Phi(w)))=y$. We set
$g_\Phi(w,y)=w+\varphi(w,z_*-\widehat \Phi(w))$ (outside the set defined by inequality \eqref{5n}
we define $g_\Phi$ to be zero, for example). Hence
$\Phi(g_\Phi(w,y))=y$. By the above, $g_\Phi(w,y)\in V$, and from \eqref{7pr} we get
$g_\Phi(w,y)\in K$. Moreover, $\|g_\Phi(w,y)-w\|_X=\|\varphi(w,z_*-\widehat
\Phi(w))\|_X \le(a(\|\Lambda\|+3)+1)r=\gamma r$, where $\gamma=a(\|\Lambda\|+3)+1$.
\end{proof}

We now formulate some assumptions and notation to be used in all the lemmas that follow.

We shall assume that the mapping  $\varphi\colon\mathbb R\times\mathbb R^n\times
\mathbb R^r\to\mathbb R^n$ (of the variables $t\in\mathbb R$, $x\in\mathbb R^n$ and $u\in\mathbb
R^r$) is continuous together with its derivative with respect to~$x$.

Recall that the set
\begin{equation*}
\mathcal A_k=\{\,\ov\alpha=(\alpha_1,\ldots,\alpha_k)\in (L_\infty([t_0,t_1]))^k :
\ov\alpha(t)\in\Sigma^k \,\,\,\text{a.~e.\ on}\,\,\, [t_0,t_1]\,\},
\end{equation*}
where $\Sigma^k=\{\,\ov\alpha=(\alpha_1,\ldots,\alpha_k)\in \mathbb R_+^k : \,\,
\sum_{i=1}^k\alpha_i=1\,\}$
was defined above for each $k\in\mathbb N$ (see the beginning of the proof of Theorem~\ref{2}).

Let a tuple $\woa=(\widehat\alpha_1,\ldots,\widehat\alpha_k)\in\mathcal A_k$
and a tuple $\wou=(\wu_1,\ldots,\wu_k)\in(L_\infty([t_0,t_1],\mathbb R^r))^k$ be fixed. Further,
let $N>k$, $\ov\alpha'=(\woa,0)\in (L_\infty([t_0,t_1]))^N$, $\ov v=(v_1,\ldots,v_{N-k})\in
(L_\infty([t_0,t_1],\mathbb R^r))^{N-k}$ and $\ov
u'=(\wu_1,\ldots,\wu_k,v_1,\ldots,v_{N-k})$.

The elements of  $\ov u'$ will be denoted by $u_i$, $i=1,\ldots,N$; that is, $u_i=\wu_i$,
$i=1,\ldots,k$,  $u_{k+i}=v_i$, $i=1,\ldots, N-k$.

\begin{lemma}[\rm the lemma on equation in variations]\label{L2}
Let $\wx$ be the solution of the differential equation
\begin{equation}\label{gor1}
\dot x=\sum_{i=1}^k\widehat\alpha_i(t)\varphi(t,x,\wu_i(t))
\end{equation}
on $[t_0,t_1]$. Then there exist  neighborhoods $\wo(\wx(t_0))$  and $\wo(\ov\alpha')$
such that, for all $\xi\in \wo(\wx(t_0))$ and $\ov\alpha=(\alpha_1,\ldots,\alpha_N)\in
\wo(\ov\alpha')$, there exists a~unique solution $x(\cdot,\xi,\ov\alpha;\ov u')$ of the Cauchy problem
\begin{equation}\label{gor2}
\dot x=\sum_{i=1}^N\alpha_i(t)\varphi(t,x,u_i(t)),\quad x(t_0)=\xi,
\end{equation}
on $[t_0,t_1]$.

The mapping $(\xi,\ov\alpha)\mapsto x(\cdot,\xi,\ov\alpha;\ov u')$ lies in
$C(\wo(\wx(t_0))\times\wo(\ov\alpha'),C([t_0,t_1],\mathbb R^n))$ and is continuously
differentiable.

If $\wx'$ is the derivative of this mapping at a~point $(\wx(t_0),\ov\alpha')$, then, for
any $\xi\in\mathbb R^n$ and
$\ov\alpha=(\alpha_1,\ldots,\alpha_N)\in(L_\infty([t_0,t_1]))^N$, the function
$h=\wx'[\xi,\ov\alpha]$ is the solution of the equation in variations
\begin{equation}\label{s2}
\dot h=\sum_{i=1}^k\widehat\alpha_i(t)\varphi_x(t,\wx(t),\wu_i(t))h+
\sum_{i=1}^N\alpha_i(t)\varphi(t,\wx(t),u_i(t)),\qquad  h(t_0)=\xi.
\end{equation}
\end{lemma}
\begin{proof} \rm
Consider the mapping  $F\colon C([t_0,t_1],\mathbb R^n)\times\mathbb
R^n\times(L_\infty([t_0,t_1]))^N\to C([t_0,t_1],\mathbb R^n)$, which is defined for all
$t\in[t_0,t_1]$ by the formula
$$
F(x,\xi,\ov\alpha)(t)=x(t)-\xi-\int_{t_0}^t\left(\sum_{i=1}^N\alpha_i(\tau)\varphi(\tau,x(\tau),u_i(\tau))\right)d\tau.
$$
The dependence of $F$ on the fixed tuple $\ov u'$ will not be indicated.

It is easily checked that at any point $(x, \xi, {\ov\alpha})\in C([t_0,t_1],\mathbb
R^n)\times\mathbb R^n\times(L_\infty([t_0,t_1]))^N$ the mapping~$F$ has the continuous
partial derivative with respect to~$x$, which acts by the rule
\begin{equation}\label{prox}
F_{x}(x, \xi,
{\ov\alpha})[h](t) =h(t)-\int_{t_0}^t\left(\sum_{i=1}^N\alpha_i(\tau)\varphi_x(\tau,
x(\tau),u_i(\tau))\right)h(\tau)\,d\tau
\end{equation}
for all $h\in C([t_0,t_1],\mathbb R^n)$ and $t\in[t_0,t_1]$ (for details, see,
for example, \cite{ATF} and~\cite{AMT}).

The existence  and continuity of the partial derivative of $F$ with respect to the variable $(\xi,\ov\alpha)$,
which enters linearly, can be easily checked. Moreover, at each point
$(x,\xi,{\ov\alpha})$ this derivative acts by the rule
\begin{equation}\label{prok}
F_{(\xi,\ov\alpha)}(x,\xi,{\ov\alpha})[\xi,\ov\alpha](t)=-\xi-\int_{t_0}^t\left(\sum_{i=1}^N
\alpha_i(\tau)\varphi(\tau, x(\tau),u_i(\tau))\right)d\tau
\end{equation}
for all $(\xi,\ov\alpha)\in \mathbb R^n\times(L_\infty([t_0,t_1]))^N$ and $t\in[t_0,t_1]$.

Therefore, the mapping~$F$ is continuously  differentiable on $C([t_0,t_1],\mathbb
R^n)\times\mathbb R^n\times(L_\infty([t_0,t_1]))^N$.

Since $\wx$ is the solution of equation  \eqref{gor1}, we have
$F(\wx,\wx(t_0),\ov\alpha')(t)=0$, $t\in[t_0,t_1]$. Finally,
that the partial derivative of~$F$ with respect to~$x$ is invertible at the point $(\wx,\wx(t_0),\ov\alpha')$
follows from the solvability of the Cauchy problem for the corresponding linear equation for any
initial conditions.

We can employ the classical implicit function theorem  (see, for example,~\cite{Z}).
According to this theorem, there exist neighborhoods $\wo(\wx)$,  $\wo(\wx(t_0))$ and
$\wo(\ov\alpha')$ and a~continuously differentiable mapping $(\xi,\ov\alpha)\mapsto
x(\cdot,\xi,\ov\alpha;\ov u')$ from $\wo(\wx(t_0))\times\wo(\ov\alpha')$ into $\wo(\wx)$
such that $F(x(t,\xi,\ov\alpha;\ov u'),\xi,\ov\alpha)(t)=0$ for all $(\xi,\ov\alpha)\in
\wo(\wx(t_0))\times\wo(\ov\alpha')$ and $t\in[t_0,t_1]$. This is equivalent to saying that
$x(\cdot,\xi,\ov\alpha;\ov u')$ is a~(unique) solution to equation \eqref{gor2}.

By the formula for the derivative of an implicit function,
the derivative $\wx'$ of the mapping $(\xi,\ov\alpha)\mapsto x(\cdot,\xi,\ov\alpha,\ov u')$ satisfies
$F_{x}(\wx,\wx(t_0),\ov\alpha')\,\wx'=-F_{(\xi,\ov\alpha)}(\wx,\wx(t_0),\ov\alpha')$ at the
point $(\wx(t_0),\ov\alpha')$.
Substituting here the expressions from \eqref{prox} and \eqref{prok} for the derivatives at the point
$(\wx,\wx(t_0),\ov\alpha')$, we see that, for all $(\xi,\ov\alpha)\in \mathbb
R^n\times(L_\infty([t_0,t_1])^N$ ($\ov\alpha=(\alpha_1,\ldots,\alpha_N)$) and
$t\in[t_0,t_1]$, the equality is satisfied
\begin{multline*}
\wx'[\xi,\ov\alpha](t)-\int_{t_0}^t\left(\sum_{i=1}^k\widehat\alpha_i(\tau)\varphi_x(\tau,
\wx(\tau),\wu_i(\tau))\right)\wx'[\xi,\ov\alpha](\tau)\,d\tau\\=\xi+\int_{t_0}^t\left(\sum_{i=1}^N
\alpha_i(\tau)\varphi(\tau,\wx(\tau),u_i(\tau))\right)d\tau.
\end{multline*}
If we denote $h=\wx'[\xi,\ov\alpha]$, then this equality is equivalent to equation \eqref{s2}.
\end{proof}

We recall that the space $C_x^1(M\times\Sigma,\,Y)$  was defined before the statement
of the generalized implicit function theorem. The set $\mathcal A_k$, for any $k\in\mathbb N$,
and the tuple of controls $\ov u'=(u_1,\ldots, u_N)$ are defined before the formulation
of the lemma on equation in variations.

Let $L>0$. We denote by $Q_L=Q_L([t_0,t_1],\mathbb R^n)$ the class of Lipschitz
vector functions on $[t_0,t_1]$ with values in  $\mathbb R^n$ and with Lipschitz constant~$L$.

As in the previous lemma, the dependence of mappings $F$ and $F_s$ on the fixed
tuple $\ov u'$ is not indicated.

\begin{lemma}[\rm the first approximation lemma]\label{L3}
Let $M$ be a bounded set in $C([t_0,t_1],\,\mathbb R^n)$,  $\Omega$ be
a~bounded set in $\mathbb R^n$, and let $L>0$. Then the mapping $F\colon
C([t_0,t_1],\mathbb R^n)\times\mathbb R^n\times(L_\infty([t_0,t_1]))^N\to
C([t_0,t_1],\mathbb R^n)$, as defined for all $t\in[t_0,t_1]$ by the formula
\begin{equation*}
F(x,\xi,\ov\alpha)(t)=x(t)-\xi-\sum_{i=1}^N\int_{t_0}^t\alpha_i(\tau)\varphi(\tau,x(\tau),u_i(\tau))
\,d\tau,
\end{equation*}
where $\ov\alpha=(\alpha_1,\ldots,\alpha_N)$, lies in the space $C_x^1=C_x^1((M\cap
Q_L)\times\Omega\times\mathcal A_N,\,C([t_0,t_1],\mathbb R^n))$. Moreover, for any
$\ov\alpha\in\mathcal A_N$, there exists a~sequence of controls
$u_s(\ov\alpha,\ov u')\in L_\infty([t_0,t_1],\mathbb R^r)$, $s\in\mathbb N$, such that
the mappings $F_s\colon C([t_0,t_1],\mathbb R^n)\times\mathbb R^n\times\mathcal A_N\to
C([t_0,t_1],\mathbb R^n)$, as defined for all $t\in[t_0,t_1]$ by the rule
\begin{equation*}
F_s(x,\xi,\ov\alpha)(t)=x(t)-\xi-\int_{t_0}^t\varphi(\tau,x(\tau),u_s(\ov\alpha;\ov
u')(\tau)) \,d\tau,
\end{equation*}
also lie in  $C_x^1$, and besides, the sequence $F_s$ converges to~$F$ in the metric of~$C^1_x$ as $s\to\infty$.
\end{lemma}

\begin{proof} \rm
Let us show that $F\in C_x^1$. By the previous lemma, the mapping $F$ is continuous together with its
partial derivative on $C([t_0,t_1],\mathbb R^n)\times\mathbb
R^n\times(L_\infty([t_0,t_1]))^N$.

Let us now check that the mapping~$F$ and its partial derivative with respect to~$x$ are bounded on the set
$(M\cap Q_L)\times\Omega\times\mathcal A_N$.

Indeed, let $\delta>0$ be such that $M\subset B_{C([t_0,t_1],\mathbb
R^n)}(0,\delta)$, $\Omega\subset B_{\mathbb R^n}(0,\delta)$ and $\gamma=\max_{1\le i\le
N}\|u_i\|_{L_\infty([t_0,t_1],\mathbb R^r)}$. The mappings $\varphi$ and $\varphi_x$
are continuous on the compact set $K=[t_0,t_1]\times B_{\mathbb R^n}(0,\delta)\times B_{\mathbb
R^r}(0,\gamma)$. We set $C=\max\{\,|\varphi(t,x,u)| : (t,x,u)\in K\,\}$ and
$C_0=\max\{\,\|\varphi_x(t,x,u)\| : (t,x,u)\in K\,\}$. Then, for any
$(x,\xi,\ov\alpha)\in(M\cap Q_L)\times\Omega\times\mathcal A_N$, $h\in
C([t_0,t_1],\mathbb R^n)$ and $t\in[t_0,t_1]$, it can be easily shown that
$|F(x,\xi,\ov\alpha)(t)|\le 2\delta+C$ and $|F_x(x,\xi,\ov\alpha)[h](t)|\le
(1+C_0)\|h\|_{C([t_0,t_1],\,\mathbb R^n)}$ (see formula~\eqref{prox}).

So, $F\in C_x^1$.

\smallskip

For each $s\in\mathbb N$ we split the interval $[t_0,t_1]$ into $s$ subintervals
$\Delta_j(s)=[t_0+j(t_1-t_0)/s, \ t_0+(j+1)(t_1-t_0)/s]$ of length
$|\Delta_j(s)|=(t_1-t_0)/s$, $j=0,\ldots,s-1$.

We set
\begin{equation}\label{msk}
\alpha_{ij}=\frac1{|\Delta_j(s)|}\int_{\Delta_j(s)}\alpha_i(t)\,dt,\quad
i=1,\ldots,N,\quad j=0,\ldots,s-1.
\end{equation}
It is clear that $\alpha_{ij}\ge0$ and $\sum_{i=1}^N\alpha_{ij}=1$, $j=0,\ldots,s-1$.

We split each subinterval $\Delta_j(s)$ into $N$ successive subintervals
$\Delta_{ji}(s,\ov\alpha)$ of length
$|\Delta_{ji}(s,\ov\alpha)|=\alpha_{ij}|\Delta_j(s)|=\alpha_{ij}(t_1-t_0)/s$,
$i=1,\ldots, N$.

Define the  function $u_s(\ov\alpha;\ov u')$ on $[t_0,t_1]$ by the rule: $u_s(\ov\alpha;\ov
u')(t)=u_i(t)$ if $t\in \Delta_{ji}(s,\ov\alpha)$, $1\le i\le N$, $j=0,1,\ldots,s-1$
(on the end-points of the subintervals the values of the functions $u_i$ and, respectively, of the function
$u_s(\ov\alpha;\ov u)$ can be taken arbitrarily). It is clear that $u_s(\ov\alpha;\ov u')\in
L_\infty([t_0,t_1],\mathbb R^r)$ and $\|u_s(\ov\alpha;\ov u')\|_{L_\infty([t_0,t_1],\mathbb
R^r)}\le\gamma$ for all $s\in\mathbb N$.

\smallskip

We claim that  $F_s\in C_x^1$ for any  $s\in\mathbb N$. We first show that, for any
$s\in\mathbb N$, the mapping $F_s$ is continuous.

To begin with, we note that the mappings $\ov\alpha\mapsto u_s(\ov\alpha;\ov u')$, \textit{qua}
mappings from $(L_\infty([t_0,t_1]))^N$ into $L_1([t_0,t_1],\mathbb R^r)$, are continuous on
$\mathcal A_N$ uniformly with respect to $s\in\mathbb N$. For simplicity of calculations, we shall check it
in the case $N=2$ and $t_0=0$,  $t_1=1$.

Let  ${\ov\alpha_0}=(\alpha_{10},\alpha_{20})\in\mathcal A_2$ be a~fixed pair.
Setting $\alpha_0=\alpha_{10}$, we have $1-\alpha_0=\alpha_{20}$. Next, let
$\beta_{j0}=|\Delta_j(s)|^{-1}\int_{\Delta_j(s)}\alpha_0(t)\,dt$, $j=0,\ldots,s-1$.
The subinterval $\Delta_j(s)$ of length $1/s$ is split into two successive subintervals
$\Delta_{j1}(s,{\ov\alpha_0})$ and $\Delta_{j2}(s,{\ov\alpha_0})$ of length, respectively,
$\beta_{j0}/s$ and $(1-\beta_{j0})/s$.

Further, let $\ov\alpha=(\alpha, 1-\alpha)$ be a different pair from~$\mathcal A_2$ and let
$\beta_j=|\Delta_j(s)|^{-1}\int_{\Delta_j(s)}\alpha(t)\,dt$, $j=0,\ldots,s-1$. Then on
each subinterval $\Delta_j(s)$ we have
\begin{multline*}
\int_{\Delta_j(s)}|u_s(\ov\alpha;\ov u')(t)-u_s({\ov\alpha_0};\ov u')(t)|\,dt=
\biggl|\int_{\beta_{j0}/s}^{\beta_j/s}|u_1(t)-u_2(t)|\,dt\biggr|\\\le
\frac{2\gamma}{s}|\beta_j-\beta_{j0}|=\frac{2\gamma}{s}\biggl|s\int_{\Delta_j(s)}\alpha(t)\,dt
-s\int_{\Delta_j(s)}\alpha_0(t)\,dt\biggr|\\\le2\gamma\int_{\Delta_j(s)}|\alpha(t)-\alpha_0(t)|\,dt
\end{multline*}
($\gamma=\max(\|u_1\|_{L_\infty([t_0,t_1],\,\mathbb
R^r)}, \|u_2\|_{L_\infty([t_0,t_1],\,\mathbb R^r)})$).
Summing these inequalities over $j=0,\ldots,s-1$, we find that
\begin{multline*}
\int_{0}^1|u_s(\ov\alpha;\ov u')(t)-u_s({\ov\alpha_0};\ov
u')(t)|\,dt\le2\gamma\int_{0}^1|\alpha(t)-\alpha_0(t)|\,dt\\\le
2\gamma\|\alpha-\alpha_0\|_{L_\infty([0,1])},
\end{multline*}
whence the required result follows.

After making this remark, we proceed with the proof of the continuity of the mappings~$F_s$.

Let $(x_0,\xi_0,\ov\alpha_0)\in C([t_0,t_1],\mathbb R^n)\times\mathbb R^n\times\mathcal
A_N$ and $\varepsilon>0$. We set $K_1=\{\,(t,x)\in\mathbb R^{n+1} : |x-
x_0(t)|\le\delta_1,\,\,t\in[t_0,t_1]\,\}\times B_{\mathbb R^r}(0,\gamma)$. The mapping
$\varphi$ is continuous on the compact set~$K_1$. Let $C_1=\max\{\,|\varphi(t,x,u)| : (t,x,u)\in
K_1\,\}$. Since~$\varphi$ is uniformly continuous on this compact set, there exists
$0<\delta_2\le\min(\delta_1,\varepsilon)$ such that
$|\varphi(t,x_1,u_1)-\varphi(t,x_2,u_2)|<\varepsilon$ for all $(t,x_i,u_i)\in K_1$,
$i=1,2$, for which $|x_1-x_2|<\delta_2$ and $|u_1-u_2|<\delta_2$.

By the above, there exists a~neighborhood $\wo(\ov\alpha_0)$ such that if $\ov\alpha\in
\wo(\ov\alpha_0)\cap\mathcal A_N$, then $u_s(\ov\alpha;\ov u')\in U_{L_1([t_0,t_1],\mathbb
R^r)}(u_s(\ov\alpha_0;\ov u'),\varepsilon\delta_2)$ for all $s\in\mathbb N$. For each
such $\ov\alpha$ and~$s$, we set $E_{\delta_2}(\ov\alpha,s)=\{\,t\in [t_0,t_1] :
|u_s(\ov\alpha;\ov u')(t)-u_s(\ov\alpha_0;\ov u')(t)|\ge\delta_2\,\}$. Then
\begin{multline*}
\delta_2{\rm
mes}\,E_{\delta_2}(\ov\alpha,s)\le\int_{E_{\delta_2}(\ov\alpha,s)}|u_s(\ov\alpha;\ov
u')(t)-u_s(\ov\alpha_0;\ov u')(t)|\,dt\\\le \|u_s(\ov\alpha;\ov u')-u_s(\ov\alpha_0;\ov
u')\|_{L_1([t_0,t_1],\mathbb R^r)}<\varepsilon\delta_2,
\end{multline*}
and hence, ${\rm mes}\,E_{\delta_2(\ov\alpha,s)}<\varepsilon$.

Now let  $x\in U_{C([t_0,t_1],\mathbb R^n)}(x_0,\delta_2)$, $\xi\in U_{\mathbb
R^n}(\xi_0,\varepsilon)$ and $\ov\alpha\in \wo(\ov\alpha_0)\cap\mathcal A_N$. For any
$t\in[t_0,t_1]$, we have
{\allowdisplaybreaks
\begin{multline*}
|F_s(x,\xi,\ov\alpha)(t)-F_s(x_0,\xi_0,\ov\alpha_0)(t)|=\biggl|x(t)-\xi -x_0(t)+\xi_0\\
-\int_{t_0}^t\varphi(\tau, x(\tau),u_s(\ov\alpha;\ov
u')(\tau))\,d\tau+\int_{t_0}^t\varphi(\tau, x_0(\tau),u_s(\ov\alpha_0;\ov
u')(\tau))\,d\tau\biggr|\\\le|x(t)-x_0(t)|+|\xi-\xi_0|+\int_{[t_0,t_1]\setminus
E_{\delta_2}(\ov\alpha,s)}|\varphi(\tau,x(\tau),u_s(\ov\alpha;\ov
u')(\tau))\\-\varphi(\tau, x_0(\tau),u_s(\ov\alpha_0;\ov
u')(\tau))|\,d\tau+\int_{E_{\delta_2}(\ov\alpha,s)}
|\varphi(\tau,x(\tau),u_s(\ov\alpha;\ov u')(\tau))\\- \varphi(\tau,
x_0(\tau),u_s(\ov\alpha_0;\ov
u')(\tau))|\,d\tau<2\varepsilon+\varepsilon(t_1-t_0)+\varepsilon2C_1,
\end{multline*}
}

\noindent that is, the mappings $F_s$ are continuous at the point $(x_0,\xi_0,\ov\alpha_0)$ uniformly with respect to~$s$, and
therefore, this is true for any point from $C([t_0,t_1],\mathbb R^n)\times\mathbb
R^n\times\mathcal A_N$.

For any $s$, the mapping $F_s$ has the partial derivative with respect to~$x$ at any point
$(x,\xi,\ov\alpha)\in C([t_0,t_1],\mathbb R^n)\times\mathbb R^n\times\mathcal A_N$,
which acts by the rule
\begin{equation*}
F_{sx}(x,\xi,{\ov\alpha};\ov u')[h](t)=h(t)-\int_{t_0}^t\varphi_x(\tau,
x(\tau),u_s(\ov\alpha;\ov u')(\tau))h(\tau)\,d\tau
\end{equation*}
for all $h\in C([t_0,t_1],\mathbb R^n)$ and $t\in[t_0,t_1]$.  This is proved along the same lines
as the existence of the partial derivative with respect to~$x$ of the mapping~$F$
in the lemma on equation in variations.

Let us show that this derivative is continuous on $C([t_0,t_1],\mathbb R^n)\times\mathbb
R^n\times\mathcal A_N$. In other words, we need to show that if $(x_0,\xi_0,\ov\alpha_0)\in
C([t_0,t_1],\mathbb R^n)\times\mathbb R^n\times\mathcal A_N$, then for any
$\varepsilon>0$ there exist neighborhoods $\wo_1(x_0)$,  $\wo_1(\xi_0)$ and
$\wo_1(\ov\alpha_0)$ such that, for all
$(x,\xi,\ov\alpha)\in\wo_1(x_0)\times\wo_1(\xi_0)\times (\wo_1(\ov\alpha_0)\cap\mathcal
A_N)$, all $h\in C([t_0,t_1],\mathbb R^n)$, $\|h\|_{C([t_0,t_1],\mathbb R^n)}\le1$ and
all $t\in[t_0,t_1]$,
\begin{multline*}
|F_{sx}(x, \xi,{\ov\alpha})[h](t)-F_{sx}(x_0,\xi_0, {\ov\alpha_0})[h](t)|
\\=\biggl|\int_{t_0}^t \varphi_x(\tau,x(\tau),u_s(\ov\alpha;\ov
u')(\tau))h(\tau)\,d\tau\\
-\int_{t_0}^t\varphi_x(\tau,x_0(\tau),u_s(\ov\alpha_0;\ov u')(\tau))
h(\tau)\,d\tau\biggr|<\varepsilon.
\end{multline*}
However, as is easy to check, to prove this inequality one needs in essence to repeat the above arguments
related to the continuity of the mapping~$F$.

The boundedness of the mappings  $F_s$ and their partial derivatives with respect to $x$, $s\in\mathbb N$, on
$(M\cap Q_L)\times\Omega\times\mathcal A_N$ can be proved as for the mapping~$F$. This implies that
$F_s\in C_x^1$ for all $s\in\mathbb N$.

Let us now show that the sequence $F_s$ converges to~$F$ in the metric of $C^1_x$ as
$s\to\infty$.

We first show that the sequence $F_s$ converges to~$F$ in $C((M\cap
Q_L)\times\Omega\times\mathcal A_N,\,C([t_0,t_1],\mathbb R^n))$ as $s\to\infty$. In other words,
we need to show that, for any $\varepsilon>0$, there exists $s_0=s_0(\varepsilon)$
such that, for all $s\ge s_0$, all $(x,\xi,\ov\alpha)\in(M\cap
Q_L)\times\Omega\times\mathcal A_N$, and all $t\in[t_0,t_1]$, the inequality holds
\begin{equation}\label{osm1}
\biggl|\int_{t_0}^t \varphi(\tau,x(\tau),u_s(\ov\alpha;\ov
u')(\tau))\,d\tau
-\sum_{i=1}^N\int_{t_0}^t\alpha_i(\tau)\varphi(\tau,x(\tau),u_i(\tau))
\,d\tau\biggr|<\varepsilon.
\end{equation}

Let $\varepsilon>0$. It can be assumed that  $0<\varepsilon<t_1-t_0$. From the
 Luzin $C$-property
and since the Lebesgue measure is regular, it follows that there exists a~closed set
$A=A(\varepsilon)\subset [t_0,t_1]$ such that ${\rm mes}\,A>(t_1-t_0)-\varepsilon$ and
that on~$A$
the functions $u_i$, $i=1,\ldots,N$, are continuous. Moreover, there exist continuous functions
$v_i$ on $[t_0,t_1]$ such that $v_i=u_i$ on~$A$ and $\|v_i\|_{C([t_0,t_1],\mathbb
R^r)}\le\|u_i\|_{L_\infty([t_0,t_1],\mathbb R^r)}$, $i=1,\ldots,N$.

Let $(x,\xi,\ov\alpha)\in(M\cap Q_L)\times\Omega\times\mathcal A_N$. On each subinterval
$\Delta_j(s)$, $0\le j\le s-1$, it is easily checked that
{\allowdisplaybreaks
\begin{multline}\label{nmix}
\int_{\Delta_j}\varphi(t,x(t),u_s(\ov\alpha;\ov
u')(t))\,dt-\sum_{i=1}^N\int_{\Delta_j}\alpha_i(t)\varphi(t,x(t),u_i(t))\,dt\\
=\sum_{i=1}^N\int_{\Delta_{ji}}\varphi(t,x(t),u_i(t))\,dt-\sum_{i=1}^N\int_{\Delta_j}\alpha_i(t)\varphi(t,x(t),u_i(t))\,dt
\\= \sum_{i=1}^N\int_{\Delta_{ji}\setminus
A}(\varphi(t,x(t),u_i(t))-\varphi(t,x(t),v_i(t)))\,dt\\-\sum_{i=1}^N\int_{\Delta_{j}\setminus
A}\alpha_i(t)(\varphi(t,x(t),u_i(t))-\varphi(t,x(t),v_i(t)))\,dt
+\\+\sum_{i=1}^N\int_{\Delta_{ji}}\varphi(t,x(t),v_i(t))\,dt-
\sum_{i=1}^N\int_{\Delta_{j}}\alpha_i(t)\varphi(t,x(t),v_i(t))\,dt
\end{multline}
}
(for brevity we replace $\Delta_j(s)$ and $\Delta_{ji}(s,\ov\alpha)$, respectively,
by $\Delta_j$ and~$\Delta_{ji}$)

Let us now estimate the expressions on the right. It is easily checked that the sum of the norms
of the first two terms on the right in~\eqref{nmix} is at most
\begin{equation}\label{nmix1}
2C\sum_{i=1}^N{\rm mes}\,(\Delta_{ji}\setminus A)+2C\int_{\Delta_{j}\setminus
A}\left(\sum_{i=1}^N\alpha_i(t)\right)\,dt=4C\,{\rm mes}\,(\Delta_{j}\setminus A),
\end{equation}
where the number~$C$ was defined at the beginning of the proof.

Now let us estimate the difference of the two last integrals in~\eqref{nmix}.  We first proceed with
each component of this difference.

Let
$\varphi(\cdot,x,v_i)=(\varphi_1(\cdot,x,v_i),\ldots,\varphi_n(\cdot,x,v_i))^T$,
$i=1,\ldots,N$. We fix $1\le l\le n$. By the mean value theorem for the integrals,
\begin{multline}\label{nmix3}
\left|\sum_{i=1}^N\int_{\Delta_{ji}}\varphi_l(t,x(t),v_i(t))\,dt-
\sum_{i=1}^N\int_{\Delta_j}\alpha_i(t)\varphi_l(t,x(t),v_i(t))\,dt\right|\\=
\biggl|\sum_{i=1}^N\varphi_l(\xi_{i},x(\xi_{i}),v_i(\xi_{i}))\alpha_{ij}|\Delta_j|
-\sum_{i=1}^N\varphi_l(\zeta_{i},x(\zeta_{i}),v_i(\zeta_{i}))\int_{\Delta_j}\alpha_i(t)\,dt
\biggr|\\\le|\Delta_j|\sum_{i=1}^N\alpha_{ij}|\varphi_l(\xi_{i},x(\xi_{i}),v_i(\xi_{i}))
-\varphi_l(\zeta_{i},x(\zeta_{i}),v_i(\zeta_{i}))|,
\end{multline}
where $\xi_{i},\zeta_{i}\in\Delta_j$, $1\le i\le N$.

Let us now estimate the absolute value of the difference on the right.
The mapping~$\varphi$ is uniformly continuous on the compact set~$K$ (which was defined
at the beginning of the proof). Hence, there exists  $\delta_0>0$ such that
$|\varphi_l(t',x',u')-\varphi_l(t'',x'',u'')|<\varepsilon$ for all $(t',x',u')$ and
$(t'',x'',u'')$ from~$K$ for which $|t'-t''|<\delta_0$, $|x'-x''|<\delta_0$ and
$|u'-u''|<\delta_0$.

Let $s_0=s_0(\varepsilon)$ be so large that
$|\Delta_j(s_0)|<\min(\varepsilon,\delta_0,\delta_0/L)$ and $|v_i(t')-v_i(t'')|<\delta_0$,
$i=1,\ldots,N$, for $t',t''\in\Delta_j(s_0)$. Hence, if $\xi_i, \zeta_i\in
\Delta_j(s_0)$, then $|\xi_i-\zeta_i|<\delta_0$, $|x(\xi_i)-x(\zeta_i)|\le
L|\xi_i-\zeta_i|<\delta_0$ and $|v_i(\xi_i)-v_i(\zeta_i)|<\delta_0$, $i=1,\ldots,N$.
Therefore, the expression on the right in~\eqref{nmix3} is majorized by
$|\Delta_j(s)|\,\varepsilon$ for $s\ge s_0$. This estimate implies that the norm of the difference of the
last two integrals in~\eqref{nmix} is estimated from above by
$|\Delta_j(s)|\sqrt{n}\,\varepsilon$.

Let us now prove inequality \eqref{osm1}. Assume that the interval $[t_0,t]$
contains a~noninteger number of subintervals $\Delta_j(s)$. If $t<t_0+(t_1-t_0)/s$, then we see that the
norm of the difference on the left in~\eqref{osm1}  is at most $2C(t-t_0)<
2C|\Delta_0(s))|<2C\varepsilon$.

Let $t\in(t_0,t_1]$ be such that the interval  $[t_0,t]$ contains an integer number of
subintervals $\Delta_j(s)$. Summing \eqref{nmix} with respect to all such subintervals, we get the
expression under the norm sign in~\eqref{osm1}.

In view of estimate \eqref{nmix1}, the sum of the norms of the first two integrals on the right (after addition)
is at most  $4C\,{\rm mes}\,([t_0,t]\setminus A)\le4C\,{\rm mes}\,([t_0,t_1]\setminus A)< 4C\varepsilon$.

The norm of the difference of the two last integrals in~\eqref{nmix} (after addition) is at most
$(t-t_0)\sqrt{n}\,\varepsilon\le(t_1-t_0)\sqrt{n}\,\varepsilon$.

So, for all $(x,\xi,\ov\alpha)\in (M\cap Q_L)\times\Omega\times\mathcal A_N$ and $s\ge
s_0(\varepsilon)$, the expression on the left in~\eqref{osm1} for the $t$~under consideration is
at most $(4C+(t_1-t_0)\sqrt{n})\varepsilon$.

The case when the interval $[t_0,t]$ is composed of an integer number of
subintervals $\Delta_j(s)$ and an additional interval of length~$\varepsilon$ can clearly be reduced to the above cases.

So, we have proved inequality \eqref{osm1}, but with $c\,\varepsilon$ in place of~$\varepsilon$,
which, however, is immaterial, because~$c$ does not depend on~$x$, on~$\ov\alpha$, and on~$t$, and hence
the sequence $F_s$ converges to~$F$ in $C((M\cap Q_L)\times\Omega\times\mathcal
A_N,\,C([t_0,t_1],\mathbb R^n))$ as $s\to\infty$.

It remains to show that the sequence $F_{sx}$ converges to~$F_x$ in $C((M\cap
Q_L)\times\Omega\times\mathcal A_N,\,C([t_0,t_1],\mathbb R^n))$ as $s\to\infty$. This means that,
for any $\varepsilon>0$, there exists $s_0=s_0(\varepsilon)$ such that
\begin{equation*}
\biggl|\int_{t_0}^t \varphi_x(\tau,x(\tau),u_s(\ov\alpha;\ov
u')(\tau))h(\tau)\,d\tau
-\sum_{i=1}^N\int_{t_0}^t\alpha_i(\tau)\varphi_x(\tau,x(\tau),u_i(\tau))
h(\tau)\,d\tau\biggr|<\varepsilon
\end{equation*}
for all  $s\ge s_0$, all $h\in C([t_0,t_1],\mathbb R^n)$, $\|h\|_{C([t_0,t_1],\mathbb
R^n)}\le1$, all $(x,\xi,\ov\alpha)\in(M\cap Q_L)\times\Omega\times\mathcal A_N$ and all
$t\in[t_0,t_1]$.
But it is easily seen that this inequality can be proved by the same line of arguments
as  in the proof of~\eqref{osm1}. This proves the first approximation lemma.
\end{proof}

We recall that the functions $u_s(\ov\alpha;\ov u')$ are defined in Lemma~\ref{L3},
and the space
$C(\mathcal M,\,Z)$ is defined before Corollary~\ref{s3}.

\begin{lemma}[\rm the second approximation lemma]\label{L4}
Let $\wx$ and $x(\cdot,\xi,\ov\alpha;\ov u')$, where $(\xi,\ov \alpha)\in
\wo(\wx(t_0))\times\wo(\ov\alpha')$,  are solutions to, respectively, equations \eqref{gor1} and
\eqref{gor2} from the lemma on equation in variations. Then there exist neighborhoods
$\wo_0(\wx(t_0))\subset\wo(\wx(t_0))$ and $\wo_0(\ov\alpha')\subset\wo(\ov\alpha')$ such that,
for all $(\xi,\ov\alpha)\in\mathcal M=
\wo_0(\wx(t_0))\times(\wo_0(\ov\alpha')\cap\mathcal A_N)$ and sufficiently large
$s\in\mathbb N$,
there exists a~unique solution $x_s(\cdot,\xi,\ov\alpha;\ov u')$
to the equation
\begin{equation}\label{difm}
\dot x=\varphi(t,x,u_s(\ov\alpha;\ov u')(t)),\quad x(t_0)=\xi,
\end{equation}
on $[t_0,t_1]$. Moreover, the mapping $(\xi,\ov\alpha)\mapsto x_s(\cdot,\xi,\ov\alpha;\ov
u')$ lies in the space $C(\mathcal M$, $C([t_0,t_1],\mathbb R^n))$ and converges in
this space to the mapping $(\xi,\ov\alpha)\mapsto x(\cdot,\xi,\ov\alpha;\ov u')$
as $s\to\infty$.
\end{lemma}
\begin{proof} \rm

Here we employ Corollary~\ref{s3} to the above generalized implicit function theorem.
We first require some preliminary considerations.

Let $\delta>0$, $\gamma=\max_{1\le i\le N}\|u_i\|_{L_\infty([t_0,t_1],\mathbb R^r)}$ and
$K_0=\{\,(t,x)\in \mathbb R\times\mathbb R^n :
|x-\wx(t)|\le\delta,\,\,\,t\in[t_0,t_1]\,\}\times B_{\mathbb R^r}(0,\gamma)$. We set
$C_0=\max\{\,|\varphi(t,x,u)| : (t,x,u)\in K_0\,\}$ and $C_1=\max\{\,\|\varphi_x(t,x,u)\| \,:\,
(t,x,u)\in K_0\,\}$

Let $F$ and $F_s$, $s\in\mathbb N$,  be the mappings  from Lemma \ref{L3}. We set
$\Lambda=F_x(\wx,\wx(t_0),\ov\alpha')$. The operator $\Lambda$ is invertible (see Lemma~\ref{L2}).

Let $x\in U_{C([t_0,t_1],\mathbb R^n)}(\wx,\delta)$, $\xi\in U_{\mathbb
R^n}(\wx(t_0),\delta)$, $\ov\alpha\in\mathcal A_N$ and $s\in\mathbb N$. Then, for all
such  $x$, $\xi$, $\ov\alpha$, $s$ and $t\in[t_0,t_1]$, we have
\begin{multline}\label{otc}
|x(t)-(\Lambda^{-1}F_s(x,\xi,\ov\alpha))(t)|\le\delta+\|\wx\|_{C([t_0,t_1],\mathbb
R^n)}\\+ \|\Lambda^{-1}\|(\delta+\|\wx\|_{C([t_0,t_1],\mathbb R^n)}
+\delta+|\wx(t_0)|+(t_1-t_0)C_0).
\end{multline}
We denote by $D$ the constant on the right and define $L=C_1(D+\delta)+C_0$.

Recall that $Q_L$ is the class of Lipschitz vector functions on $[t_0,t_1]$ with values in~$\mathbb R^n$
and with Lipschitz constant~$L$. It is easily checked that $Q_L$ is a~convex
closed set in $C([t_0,t_1],\,\mathbb R^n)$.

By Lemma~\ref{L3} the mappings $F$ and $F_s$, $s\in\mathbb N$, are contained in the space $\widehat C^1_x=C^1_x((U_{C([t_0,t_1],\,\mathbb R^n)}(\wx,\delta)\cap
Q_L)\times U_{\mathbb R^n}(\wx(t_0),\delta)\times\mathcal A_N,\,C([t_0,t_1],\,\mathbb
R^n))$ and converge to~$F$ in this space $F_s$ as $s\to\infty$.

Now we can apply Corollary~\ref{s3} to the generalized implicit function theorem, in which
$X=Y=C([t_0,t_1],\,\mathbb R^n)$, $\Sigma=U_{\mathbb R^n}(\wx(t_0),\delta)\times
\mathcal A_N$, $\ws=(\wx(t_0), \ov\alpha')$, $\wx=\wx\cd$, $V=U_{C([t_0,t_1],\,\mathbb
R^n)}(\wx,\delta)$, $Q=Q_{L}$ and $\wF=F$.

From Lemma~\ref{L2} it follows that  $F(\wx,\wx(t_0),\ov\alpha')=0$, and moreover, the
operator $\Lambda=F_x(\wx,\wx(t_0),\ov\alpha')$ is invertible, as was noted above.

Let neighborhoods $V'_0\subset V_0\subset V$ of the point $\wx$, a~neighborhood $U_0\subset
U_{\mathbb R^n}(\wx(t_0),\delta)\times\mathcal A_N$ of the point $(\wx(t_0),\ov\alpha')$ and
a~neighborhood~$W_0$ of the mapping~$F$ be from the conclusion of the theorem.

Since the mappings  $F_s$ converge to~$F$ in the space $\widehat C^1_x$ as
$s\to\infty$, there exists $s_0$ such that $F_s\in W_0$ for all $s\ge s_0$.

Let us check that $x-\Lambda^{-1}F_s(x,\xi,\ov\alpha)\in Q_{L}$ for all
$(x,\xi,\ov\alpha)\in(V'_0\cap Q_{L})\times U_0$ and $s\in\mathbb N$. Indeed, if
$y=x-\Lambda^{-1}F_s(x,\xi,\ov\alpha)$, then $\Lambda y=\Lambda x-F_s(x,\xi,\ov\alpha)$,
or (by the definition of~$\Lambda$ and $F_s(x,\xi,\ov\alpha)$)
\begin{multline}\label{lip0}
y(t)-\int_{t_0}^t\left(\sum_{i=1}^k\widehat\alpha_i(\tau)\varphi_x(\tau,\wx(\tau),\wu_i(\tau))\right)y(\tau)\,d\tau
=x(t)\\-\int_{t_0}^t\left(\sum_{i=1}^k\widehat\alpha_i(\tau)\varphi_x(\tau,\wx(\tau),\wu_i(\tau))\right)x(\tau)\,d\tau
-x(t)+\xi\\+\int_{t_0}^t\varphi(\tau,x(\tau),u_s(\ov\alpha;\ov u')(\tau)) \,d\tau
\end{multline}
for all $t\in[t_0,t_1]$.

Therefore, since $\|y\|_{C([t_0,t_1],\,\mathbb R^n)}\le D$ (see \eqref{otc}),
$\|x\|_{C([t_0,t_1],\,\mathbb R^n)}\le \delta$, $\sum_{i=1}^k\widehat\alpha_i(\tau)=1$
for almost all $t\in[t_0,t_1]$ and $\|u_s(\ov\alpha;\ov u')\|_{L_\infty([t_0,t_1],\,\mathbb
R^r)}\le\gamma$ for all $s\in\mathbb N$ (see Lemma~\ref{3}), we see that, for any
$t',t''\in[t_0,t_1]$,
\begin{multline*}
|y(t')-y(t'')|\le\biggl|\int_{t'}^{t''}\sum_{i=1}^k\widehat\alpha_i(\tau)\|\varphi_x(\tau,\wx(\tau),\wu_i(\tau))\|
|y(\tau)|\,d\tau\biggr|\\+\biggl|\int_{t'}^{t''}\sum_{i=1}^k\widehat\alpha_i(\tau)\|\varphi_x(\tau,\wx(\tau),\wu_i(\tau))\|
|x(\tau)|\,d\tau\biggr|\\+\biggl|\int_{t'}^{t''}|\varphi(\tau,\wx(\tau),u_s(\ov\alpha;\ov
u')(\tau))|\,d\tau\biggr|\le(C_1D+C_1\delta+C_0)|t'-t''|=L|t'-t''|
\end{multline*}
and so, $x-\Lambda^{-1}F_s(x,\xi,\ov\alpha)\in Q_{L}$ for all
$(x,\xi,\ov\alpha)\in(V'_0\cap Q_{L})\times U_0$ and $s\in\mathbb N$.

The same argument shows that $x-\Lambda^{-1}\wF(x,\xi,\ov\alpha)\in Q_{L}$ for all
$(x,\xi,\ov\alpha)\in(V'_0\cap Q_{L})\times U_0$ (instead of the last integral on the right
in~\eqref{lip0}, we have the integral
$\int_{t_0}^t\left(\sum_{i=1}^N\alpha_i(\tau)\varphi(\tau,x(\tau),u_i(\tau))\right)d\tau$).

Hence by Corollary~\ref{s3}, for all $s\ge s_0$, there exist
continuous mappings $g_{F_s}\colon U_0\to V_0\cap Q_{L}$ and $g_F\colon U_0\to V_0\cap
Q_{L}$ such that  $F_s(g_{F_s}(\xi,\ov\alpha),\xi,\ov\alpha)(t)=0$ and
$F(g_F(\xi,\ov\alpha),\xi,\ov\alpha)(t)=0$ for all $(\xi,\ov\alpha)\in U_0$ and
$t\in[t_0,t_1]$.

This is equivalent to saying that, for all $(\xi,\ov\alpha)\in U_0$, the function
$g_{F_s}(\xi,\ov\alpha)$ is a~unique solution $x_s(\cdot,\xi,\ov\alpha;\ov u)$
to equation \eqref{difm} and the function $g_{\wF}(\xi,\ov\alpha)$ is a~unique solution
$x(\cdot,\xi,\ov\alpha;\ov u)$ to equation \eqref{gor2}, whose properties are described in Lemma~\ref{L2}.

Moreover, by Corollary \ref{s3} there exists a~neighborhood $U'_0\subset U_0$ of the point
$(\wx(t_0),\ov\alpha')$ (it can be assumed that this neighborhood has the form $\mathcal
M=\wo_0(\wx(t_0))\times(\wo_0(\ov\alpha')\cap\mathcal A_N)$, where
$\wo_0(\wx(t_0))\subset\wo(\wx(t_0))$, $\wo_0(\ov\alpha')\subset \wo(\ov\alpha')$, and
$\wo(\wx(t_0))$ and $\wo(\ov\alpha')$ are the neighborhoods from Lemma \ref{L2}) such that
\begin{equation*}
\|x_s-x\|_{C(\mathcal M,\, C([t_0,t_1],\,\mathbb
R^n))}\le2\|\Lambda^{-1}\|\|F_s-F\|_{C((V\cap Q_{L})\times\Sigma,\, C([t_0,t_1],\,\mathbb
R^n))},
\end{equation*}
where $x_s$ and $x$ are, respectively, the continuous mappings $(\xi,\ov\alpha)\mapsto
x_s(\cdot,\xi,\ov\alpha;\ov u')$ and $(\xi,\ov\alpha)\mapsto x(\cdot,\xi,\ov\alpha;\ov u')$,
and $\Lambda=F_x(\wx,\wx(t_0),\ov\alpha')$.

The quantity on the right tends to zero as $s\to\infty$, and hence, $x_s\to x$ as
$s\to\infty$ in the metric of the space $C(\mathcal M,\, C([t_0,t_1],\,\mathbb R^n))$.
\end{proof}

\end{document}